\documentclass[a4paper, 11pt, reqno]{amsart}

\usepackage[a4paper]{geometry}
\usepackage[utf8]{inputenc}
\usepackage[english]{babel}
\usepackage[T1]{fontenc}
\usepackage{textcomp}
\usepackage{lmodern}
\usepackage{graphicx}

\usepackage{amsaddr}

\usepackage{todonotes}

\usepackage[refpage]{nomencl}
\makenomenclature

\usepackage{amsmath}
\usepackage{amsfonts}
\usepackage{amsthm}
\usepackage{latexsym}
\usepackage{amssymb}
\usepackage[all,cmtip]{xy}

\DeclareMathOperator{\Ho}{Ho}
\DeclareMathOperator{\Natdg}{Nat_{dg}}

\DeclareMathOperator{\Hom}{Hom}

\DeclareMathOperator{\Fundg}{Fun_{dg}}
\DeclareMathOperator{\Ob}{Ob}

\DeclareMathOperator{\cone}{C}

\DeclareMathOperator{\RHom}{\mathbb{R}\underline{Hom}}

\DeclareMathOperator{\Spec}{Spec}
\DeclareMathOperator{\coker}{coker}

\newcommand{\cat}{\mathbf}

\newcommand{\opp}[1]{{#1}^{\mathrm{op}}}
\newcommand{\kat}{\mathsf}

\newcommand{\ldual}{L}
\newcommand{\rdual}{R}

\newcommand{\lder}{\mathbb L}
\newcommand{\rder}{\mathbb R}
\newcommand{\lotimes}{\otimes^\lder}
\newcommand{\ldiamond}{\diamond^\lder}
\newcommand{\qis}{\overset{\mathrm{qis}}{\approx}}
\newcommand{\qe}{\overset{\mathrm{qe}}{\approx}}

\newcommand{\tria}[1]{\mathrm{tria}(#1)}

\newcommand{\per}[1]{\mathrm{per}(#1)}

\newcommand{\isorightarrow}{\xrightarrow{\sim}}

\newcommand{\profto}{\rightsquigarrow}
\newcommand{\basering}[1]{\mathbf{#1}}
\newcommand{\enrich}[1]{\underline{#1}}

\newcommand{\Mod}[1]{\mathsf{Mod}(#1)}

\newcommand{\comp}[1]{\mathsf{C}(#1)}
\newcommand{\hocomp}[1]{\mathsf{K}(#1)}
\newcommand{\dercomp}[1]{\mathsf{D}(#1)}
\newcommand{\compdg}[1]{\mathsf{C}_\mathrm{dg}(#1)}

\newcommand{\acy}[1]{\mathrm{Ac}(#1)}
\newcommand{\hproj}[1]{\mathrm{h\textrm{-}proj}(#1)}
\newcommand{\hinj}[1]{\mathrm{h\textrm{-}inj}(#1)}

\newcommand{\rep}[1]{\mathrm{rep}(#1)}
\newcommand{\horep}[1]{\mathrm{hrep}(#1)}
\newcommand{\qrep}[1]{\mathrm{qrep}(#1)}
\newcommand{\rrep}[1]{\mathrm{rep}^r(#1)}
\newcommand{\lrep}[1]{\mathrm{rep}^l(#1)}
\newcommand{\rhorep}[1]{\mathrm{hrep}^r(#1)}
\newcommand{\lhorep}[1]{\mathrm{hrep}^l(#1)}
\newcommand{\rqrep}[1]{\mathrm{qrep}^r(#1)}
\newcommand{\lqrep}[1]{\mathrm{qrep}^l(#1)}

\newtheorem{thm}{Theorem}[section]

\newtheorem{prop}[thm]{Proposition}
\newtheorem{coroll}[thm]{Corollary}

\newtheorem{lemma}[thm]{Lemma}

\theoremstyle{remark}
\newtheorem{remark}[thm]{Remark}
\newtheorem{example}[thm]{Example}

\theoremstyle{definition}
\newtheorem{defin}[thm]{Definition}

\numberwithin{equation}{section}

\entrymodifiers={+!!<0pt,\fontdimen22\textfont2>}   % per xypic

\title{Adjunctions of quasi-functors between dg-categories}
\author{Francesco Genovese}
\address{Dipartimento di Matematica ``F. Casorati'', Università di Pavia, \\ Via Ferrata 5, 27100 Pavia (PV), Italy}
\email{francesco.genovese01@ateneopv.it}


\begin{document}
\begin{abstract}
We study right quasi-representable differential graded bimodules as quasi-functors between dg-categories. We prove that a quasi-functor has a left adjoint if and only if it is left quasi-representable.
\end{abstract}

\maketitle
\section{Introduction}
\emph{Differential graded (dg-) categories} are categories enriched over the closed symmetric monoidal category $\comp{\basering k}$ of complexes of $\basering k$-modules ($\basering k$ is a fixed ground commutative ring). They carry a significant homotopical structure, induced from that of cochain complexes, and they are a popular incarnation of higher categories. In particular, \emph{pretriangulated dg-categories} are employed as enhancements for triangulated categories, overcoming their well-known technical issues such as the non-functoriality of cones. 

Being studied as enriched categories, dg-categories admit obvious morphisms between them, namely, \emph{dg-functors}. They are simply defined as functors which preserve the cochain complex structure of hom-sets. However, dg-functors don't retain the relevant homotopical structure of dg-categories, and must be replaced with more complicated morphisms, which are called \emph{quasi-functors}. To be more precise, it has been proved that the category $\kat{dgCat}$ of dg-categories and dg-functors admits a model category structure whose weak equivalences are the quasi-equivalences (see \cite{tabuada-dgcat}); moreover, its homotopy category $\kat{Hqe} = \Ho(\kat{dgCat})$ has a natural structure of closed symmetric monoidal category (see \cite{toen-morita}). The internal hom in $\kat{Hqe}$ between dg-categories $\cat A$ and $\cat B$ is denoted by $\RHom(\cat A, \cat B)$, and it is the ``dg-category of quasi-functors'', defined up to quasi-equivalence.

The dg-category $\RHom(\cat A, \cat B)$ has many ``incarnations''. One of them, which we will employ in this work, involves particular \emph{dg-bimodules}. Given dg-categories $\cat A$ and $\cat B$, an $\cat A$-$\cat B$-dg-bimodule is a dg-functor $\opp{\cat B} \otimes \cat A \to \compdg{\basering k}$, where $\compdg{\basering k}$ is the dg-category of complexes of $\basering k$-modules. It can also be viewed as a dg-functor $\cat A \to \compdg{\cat B}$, where $\compdg{\cat B} = \Fundg(\opp{\cat B}, \compdg{\basering k})$ is the dg-category of right $\cat B$-dg-modules, namely, the dg-category of dg-functors $\opp{\cat B} \to \compdg{\basering k}$. Any dg-functor $F \colon \cat A \to \cat B$ defines a dg-bimodule
\begin{equation*}
h_F = \cat B(-,F(-)) \colon \cat A \to \compdg{\cat B},
\end{equation*}
which has the property of being \emph{right representable}: for all $A \in \cat A$, $\cat B(-,F(A))$ is the right $\cat B$-module represented by $F(A) \in \cat B$. Conversely, if $T$ is a $\cat A$-$\cat B$-bimodule such that $T(A) \cong \cat B(-,F(A))$ as right $\cat B$-modules for some $F(A) \in \cat B$, then $T$ can be identified with a dg-functor $\cat A \to \cat B$. Now, the idea is to weaken the right representability hypothesis: an $\cat A$-$\cat B$-dg-bimodule $T$ is \emph{right quasi-representable} if the right $\cat B$-dg-module $T(A)$ is \emph{quasi-isomorphic} to $\cat B(-,F(A))$ for some $F(A) \in \cat B$, for all $A \in \cat A$ (quasi-isomorphisms of dg-modules are defined as morphisms which are objectwise quasi-isomorphisms of complexes). The category of right quasi-representable bimodules $\cat A \to \cat B$ is denoted by $\rqrep{\cat A, \cat B}$ (see Subsection \ref{subsec:duality_bimod} and the discussion around Proposition \ref{prop:lqrep_rqrep}). It is easily seen that a right quasi-representable $\cat A$-$\cat B$-bimodule does \emph{not} induce a dg-functor $\cat A \to \cat B$; however, it induces a functor between the homotopy categories: $H^0(\cat A) \to H^0(\cat B)$. It turns out that such bimodules are a possible incarnation of quasi-functors (and, with a slight abuse of terminology, will be called themselves quasi-functors); the precise statement goes as follows:
\begin{prop}[{\cite[Theorem 4.5]{keller-dgcat}}] \label{prop:rep_quasifun}
The category $H^0(\RHom(\cat A, \cat B))$ is naturally equivalent to the category $\rqrep{\cat A, \cat B}$ of right quasi-representable $\cat A$-$\cat B$-dg-bimodules.
\end{prop}

As hinted, quasi-functors are the homotopy relevant morphisms between dg-categories. Being particular bimodules, they are the $1$-morphisms of a \emph{bicategory}; in that context, there is a natural notion of \emph{adjunction} (see Section \ref{section:bicat_bimod_adj}). What we have managed to do in this work is to give a simple characterisation of adjoint quasi-functors, in term of quasi-representability. To have a grasp of the idea, start from an ordinary adjunction of dg-functors $F \dashv G \colon \cat A \to \cat B$: it is an isomorphism
\begin{equation*}
\cat B(F(A),B) \cong \cat A(A,G(B)),
\end{equation*}
natural in $A \in \cat A$ and $B \in \cat B$. This naturality implies that this is actually an isomorphism of bimodules: on the right hand side we have the $\cat B$-$\cat A$-bimodule $h_G$, whereas on the left hand side we have the bimodule $h^F$, which is obtained from $h_F \colon \cat A \to \compdg{\cat B}$ by means of a sort of \emph{duality}:
\begin{align*}
h_F(A)(B) = \cat B(B,F(A)), \\
h^F(B)(A) = \cat B(F(A),B)).
\end{align*}
This duality is actually defined for all bimodules: it maps functorially $\cat A$-$\cat B$-bimodules to $\cat B$-$\cat A$-bimodules, and vice-versa (see Proposition \ref{prop:cap3_duality_bimodules}). Bimodules of the form $h_F$ (up to isomorphism of bimodules), namely right representable bimodules, are mapped to bimodules of the form $h^F$ (up to isomorphism), which are called \emph{left representable}. Clearly, saying ``the dg-functor $G$ has a left adjoint'' is equivalent to saying ``$h_G$ is left representable'': $h^F \cong h_G$. Upon overcoming some technical difficulties (addressed by means of the duality construction we have mentioned) it can be proved that this characterisation can be consistently extended to quasi-functors. Defining \emph{left quasi-representable bimodules} in the obvious way, we obtain:
\begin{thm}[Proposition \ref{prop:quasifun_adjoint_repr}] \label{thm:quasifun_adjoint_repr}
Let $G \colon \cat B \to \cat A$ be a quasi-functor. Then, $G$ has a left adjoint quasi-functor if and only if it is left quasi-representable.
\end{thm}
As the reader may expect, there is a similar characterisation of right adjoints. The result can be applied to prove an existence theorem of adjoint quasi-functors, under some hypotheses on the dg-categories:
\begin{thm}[Theorem \ref{thm:quasifun_adjoint_existence}]
Let $\cat A, \cat B$ be dg-categories. Assume that $\cat A$ is triangulated and smooth, and that $\cat B$ is locally perfect. Let $T \colon \cat A \to \cat B$ be a quasi-functor. Then, $T$ admits both a left and a right adjoint.
\end{thm}
A key theoretical tool we use to establish Theorem \ref{thm:quasifun_adjoint_existence} is \emph{end and coend calculus} (developed in Section \ref{chapter:dgends}). A (co)end of a dg-(endo)module $F \colon \opp{\cat A} \otimes \cat A \to \compdg{\basering k}$ is essentially a complex which (co)equalises the left and right actions of $\cat A$ on $F$. Some relevant categorical constructions can be expressed with ends and coends, for example the complex of dg-natural transformations between two dg-functors $F, G \colon \cat A \to \cat B$ is an end (see Corollary \ref{coroll:nattrans_end}):
\begin{equation*}
\Natdg(F,G) \cong \int_A \cat B(F(A),G(A)).
\end{equation*}
With (co)end calculus, many categorical results become a matter of straightforward computations. The author personally hopes that this tool becomes more widespread among mathematicians who employ categorical methods in their work.
\subsection*{Acknowledgements}
This article has been estracted from the author's PhD thesis, written under the supervision of Prof. Alberto Canonaco. The author also thanks Fosco Loregian for teaching him the formalism of (co)ends in category theory.
\section{Preliminaries on dg-categories}
This section contains the basic well-known facts about dg-categories; we refer to \cite{keller-dgcat} for a comprehensive survey. We fix, once and for all, a ground commutative ring $\basering k$. Virtually every category we shall encounter will be at least $\basering k$-linear, so we allow ourself some sloppiness, and often employ the terms ``category'' and ``functor'' meaning ``$\basering k$-category'' and ``$\basering k$-functor''. 
\subsection{Dg-categories and dg-functors}
A dg-category is a category enriched over the closed symmetric monoidal category $\comp{\basering k}$ of cochain complexes of $\basering k$-modules:
\begin{defin} \label{def:cap1_dgcat}
A \emph{differential graded (dg-) category} $\cat A$ consists of a set of objects $\Ob \cat A$, a hom-complex $\cat A(A,B)$ for any couple of objects $A,B$, and (unital, associative) composition chain maps of complexes of $\basering k$-modules:
\begin{equation*}
\begin{split}
\cat A(B,C) \otimes \cat A(A,B) & \to \cat A(A,C), \\
g \otimes f & \mapsto gf = g \circ f
\end{split}
\end{equation*}
\end{defin}
\begin{defin} \label{def:cap1_dgfun}
Let $\cat A$ and $\cat B$ be dg-categories. A \emph{dg-functor} $F$ consists of the following data:
\begin{itemize}
\item a function $F \colon \Ob \cat A \to \Ob \cat B$;
\item for any couple of objects $(A,B)$ of $\cat A$, a chain map
\begin{equation*}
F = F_{(A,B)} \colon \cat A(A,B) \to \cat B(F(A),F(B)),
\end{equation*}
\end{itemize}
subject to the usual associativity and unitality axioms.
\end{defin}
\begin{example}
An example of dg-category is given by the \emph{dg-category of complexes} $\compdg{\basering k$}: it has the same objects as $\comp{\basering k}$, and complexes of morphisms $\enrich{\Hom}(V,W)$ given by:
\begin{equation} \label{eq:internalhom_complex_def}
\begin{split}
\enrich{\Hom}(V,W)^n &= \prod_{i \in \mathbb Z} \Hom(V^i, W^{i+k}), \\
df &= d_W \circ f - (-1)^{|f|} f \circ d_V.
\end{split}
\end{equation}
There is a natural ($\basering k$-linear) bijection
\begin{equation} \label{eq:internalhom_complex_1}
\Hom(Z \otimes V, W)  \isorightarrow \Hom(Z, \enrich{\Hom}(V,W)).
\end{equation}
\end{example}
\begin{remark}
All usual categorical constructions can be carried out for dg-categories and dg-functors. 
\begin{enumerate}
\item Any ordinary ($\basering k$-linear) category can be viewed as a dg-category, with trivial complexes of morphisms.
\item For any dg-category $\cat A$ there is the \emph{opposite dg-category} $\opp{\cat A}$, such that
\begin{equation*}
\opp{\cat A}(A,B) = \cat A(B,A), 
\end{equation*}
with the same compositions as in $\cat A$ up to a sign determined by the Koszul sign rule:
\begin{equation*}
\opp{f} \opp{g} = (-1)^{|f||g|} \opp{(gf)},
\end{equation*}
denoting by $\opp{f} \in \opp{\cat A}(B,A)$ the corresponding morphism of $f \in \cat A(A,B)$.
\item Given dg-categories $\cat A$ and $\cat B$, there is the \emph{tensor product} $\cat A \otimes \cat B$: its objects are couples $(A,B)$ where $A \in \cat A$ and $B \in \cat B$; its hom-complexes are given by
\begin{equation*}
(\cat A \otimes \cat B) ((A,B),(A',B'))= \cat A(A,A') \otimes \cat B(B,B').
\end{equation*}
Compositions of two morphisms $f \otimes g$ and $f' \otimes g'$ is given by:
\begin{equation*}
(f' \otimes g')(f \otimes g) = (-1)^{|g'||f|} f'f \otimes g'g.
\end{equation*}
The tensor product commutes with taking opposites: $\opp{(\cat A \otimes \cat B)} = \opp{\cat A} \otimes \opp{\cat B}$. Also, it is symmetric, namely, there is an isomorphism of dg-categories: $\cat A \otimes \cat B \cong \cat B \otimes \cat A$.
\item Given dg-categories $\cat A$ and $\cat B$, there is a dg-category $\Fundg(\cat A, \cat B)$ whose objects are dg-functors $\cat A \to \cat B$ and whose complexes of morphisms are the so-called \emph{dg-natural transformations}:
A dg-natural transformation $\varphi \colon F \to G$ of degree $p$ is a collection of degree $p$ morphisms
\begin{equation*}
\varphi_A \colon F(A) \to G(A),
\end{equation*}
for all $A \in \cat A$, such that for any degree $q$ morphism $f \in \cat A(A,A')$ the following diagram is commutative up to the sign $(-1)^{|p||q|}$:
\begin{equation*}
\xymatrix{
F(A) \ar[r]^{\varphi_A} \ar[d]_{F(f)}  & G(A) \ar[d]^{G(f)} \\
F(A') \ar[r]^{\varphi_{A'}} & G(A').
}
\end{equation*}
Differentials and compositions of dg-natural transformations are defined objectwise. The complex of dg-natural transformations $F \to G$ will often be denoted by $\Natdg(F,G)$.

The dg-category $\Fundg(\cat A, \cat B)$ characterises the internal hom in the monoidal category $(\kat{dgCat}, \otimes)$ of (small) dg-categories and dg-functors. Namely, there is a natural isomorphism in $\kat{dgCat}$:
\begin{equation} \label{eq:sec1_iso_fundg_internalhom}
\Fundg(\cat A \otimes \cat B, \cat C) \cong \Fundg(\cat A, \Fundg(\cat B, \cat C)).
\end{equation}
\item Dg-functors $\cat A \otimes \cat B \to \cat C$ are called \emph{dg-bifunctors}, and they are ``dg-functors of two variables $A \in \cat A$ and $B \in \cat B$'', separately dg-functorial in both. The same is true in general for dg-functors $\cat A_1 \otimes \cdots \otimes \cat A_n \to \cat C$: they can be viewed as ``dg-functors of many variables'', with functoriality in each variable. Sometimes, we will employ \emph{Einstein notation} to indicate which variables are covariant and which ones are contravariant. For example, a dg-functor $F \colon \opp{\cat B} \otimes  \cat A_1 \otimes \cat A_2  \to \cat C$ will be written as
\begin{equation*}
F(B, A_1, A_2) = F_{A_1,A_2}^B; \nomenclature{$F^B_{A_1, A_2}$}{Einstein notation for $F(B,A_1,A_2)$, meaning that the functor $F$ is covariant in $A_1, A_2$ and contravariant in $B$}
\end{equation*}
lower variables are covariant, whereas upper variables are contravariant. Moreover, we shall set (for instance)
\begin{equation*}
F^f_{A_1, A_2} = F(f \otimes 1_{A_1} \otimes 1_{A_2}),
\end{equation*}
for any morphism $f$.
\end{enumerate}
\end{remark}

The operations of taking cocycles and cohomology can be extended from complexes of $\basering k$-modules to dg-categories and dg-functors:
\begin{defin}
Let $\cat A$ be a dg-category. The \emph{underlying category} (resp. the \emph{homotopy category}) of $\cat A$ is the category $Z^0(\cat A)$ (resp. $H^0(\cat A)$) which is defined as follows:
\begin{itemize}
\item $\Ob Z^0(\cat A) = \Ob H^0(\cat A) = \Ob \cat A$,
\item $Z^0(\cat A)(A,B) = Z^0(\cat A(A,B))$ (respectively $H^0(\cat A)(A,B) = H^0(\cat A(A,B))$), for all $A, B \in \cat A$,
\end{itemize}
with natural compositions and identities.
\end{defin}
The mappings $\cat A \mapsto Z^0(\cat A)$ and $\cat A \mapsto H^0(\cat A)$ are functorial: given a dg-functor $F \colon \cat A \to \cat B$, there are natural induced functors 
\begin{align*}
Z^0(F) \colon Z^0(\cat A) & \to Z^0(\cat B), \\
H^0(F) \colon H^0(\cat A) & \to H^0(\cat B). 
\end{align*}
Given two objects $A,B$ in a dg-category $\cat A$, we say that they are \emph{dg-isomorphic} (resp. \emph{homotopy equivalent}), and write $A \cong B$ (resp. $A \approx B$) if they are isomorphic in $Z^0(\cat A)$ (resp. $H^0(\cat A)$).

A \emph{quasi-equivalence} is a dg-functor $F \colon \cat A \to \cat B$ such that the maps
\begin{equation*}
F_{(A,B)} \colon \cat A(A,B) \to \cat B(F(A),F(B))
\end{equation*}
are quasi-isomorphisms, and $H^0(F)$ is essentially surjective. Given dg-categories $\cat A$ and $\cat B$, we say that they are \emph{quasi-equivalent}, writing $\cat A \qe \cat B$, \nomenclature{$\cat A \qe \cat B$}{Dg-categories $\cat A, \cat B$ are quasi-equivalent} if there exists a zig-zag of quasi-equivalences:
\begin{equation*}
\cat A \leftarrow \cat A_1 \rightarrow \ldots \leftarrow \cat A_n \rightarrow \cat B.
\end{equation*}
\subsection{Dg-modules and bimodules} 
Dg-functors with values in the dg-category of complexes are called \emph{dg-modules} and are worth being studied in their own right.
\begin{defin} \label{def:cap1_modules}
Let $\cat A$ be a dg-category. A \emph{left $\cat A$-dg-module} is a dg-functor $\cat A \to \compdg{\basering k}$. A \emph{right $\cat A$-dg-module} is a dg-functor $\opp{\cat A} \to \compdg{\basering k}$. 

Let $\cat B$ be another dg-category. A \emph{$\cat A$-$\cat B$-dg-bimodule} is a dg-bifunctor $\opp{\cat B} \otimes \cat A \to \compdg{\basering k}$.
\end{defin}
\begin{remark} \label{remark:modules_functors_actions}
Let $F \colon \cat A \to \compdg{\basering k}$ be a left dg-module. Given a couple of objects $(A,A')$, we may view the functor $F$ on the hom-complex $\cat A(A,A')$ as an element
\begin{align*}
F_{(A,A')} \in  & \Hom(\cat A(A,A'), \enrich{\Hom}(F(A), F(A')) \\
&\cong \Hom(\cat A(A,A') \otimes F(A), F(A')).
\end{align*}
So, giving $F$ as a functor $\cat A \to \compdg{\basering k}$ is the same as giving a complex of $\basering k$-modules $F(A)$ for all objects $A \in \cat A$, and chain maps
\begin{align*}
\cat A(A,A') \otimes F(A) & \to F(A'), \\
f \otimes x \mapsto fx,
\end{align*}
such that
\begin{align*}
g(fx) &= (gf)x, \\
1_A x &= x,
\end{align*}
for any $x \in F(A)$, for any $f \colon A \to A'$ and $g \colon A' \to A''$. So, a left dg-module is given by a family of complexes parametrised by objects of $\cat A$, together with a left $\cat A$-action. By construction, we have
\begin{equation*}
fx = F(f)(x).
\end{equation*}
Similarly, we have the characterisation of a right dg-module $F \colon \opp{\cat A} \to \compdg{\basering k}$ as family of complexes with a right action:
\begin{align*}
F(A) \otimes \cat A(A',A) & \to F(A'), \\
x \otimes f \mapsto xf;
\end{align*}
in terms of $F$, we have
\begin{equation*}
xf = (-1)^{|x||f|} F(f)(x),
\end{equation*}
taking into account the Kozsul sign rule. Finally, giving a dg-bimodule $F \colon \opp{\cat B} \otimes \cat A \to \compdg{\basering k}$ is the same as giving a family of complexes $F(B,A)$ together with a left action of $\cat A$ and a right action of $\cat B$, subject to the compatibility condition:
\begin{equation*}
(gx)f = g(xf),
\end{equation*}
whenever $x \in F(B,A)$ and $f \colon B' \to B, g \colon A \to A'$. We allow ourselves to drop parentheses and write $gxf$ meaning $(gx)f=g(xf)$. Actually, in terms of the original bifunctor $F$, we have
\begin{equation*}
F(f \otimes g)(x) = (-1)^{|f|(|x|+|g|)} gxf.
\end{equation*}
In the following, we will allow ourselves to shift freely from one characterisation of dg-(bi)modules to another, and adopt indiscriminately either the ``functor'' notation or the ``left/right action'' notation, keeping in mind how to interchange them.
\end{remark}
The dg-category of right $\cat A$-modules is the category of functors $\Fundg(\opp{\cat A}, \compdg{\basering k})$, and will be denoted by $\compdg{\cat A}$ \nomenclature{$\compdg{\cat A}$}{The dg-category of right $\cat A$-dg-modules}. Moreover, we set:
\begin{align}
\comp{\cat A} &= Z^0(\compdg{\cat A}), \nomenclature{$\comp{\cat A}$}{The ordinary category of right $\cat A$-dg-modules}\\
\hocomp{\cat A} &= H^0(\compdg{\cat A}). \nomenclature{$\hocomp{\cat A}$}{The homotopy category of right $\cat A$-dg-modules}
\end{align}
A morphism of left $\cat A$-modules $\varphi \colon F \to G$ is simply a dg-natural transformation of functors. Adopting the ``left action'' notation as explained in Remark \ref{remark:modules_functors_actions}, we see that $\varphi$ can be viewed as a family of maps $\varphi_A \colon F(A) \to G(A)$ such that
\begin{equation*}
\varphi_{A'}(fx) = (-1)^{|\varphi||f|} f \varphi_A(x),
\end{equation*}
for any $x \in F(A)$ and $f \in \cat A(A, A')$ (notice the Koszul sign rule). Similarly, a morphism of right $\cat A$-modules $\psi \colon M \to N$ satisfies the following:
\begin{equation*}
\psi_A(xf) = \psi_{A'}(x)f,
\end{equation*}
for any $x \in M(A')$ and $f \in \cat A(A,A')$. Finally, a morphism of $\cat A$-$\cat B$-bimodules $\xi \colon F_1 \to F_2$ is required to satisfy both compatibilities with the left and right actions:
\begin{equation*}
\xi_{(B',A')}(gxf) = (-1)^{|g||\xi|} g \xi_{(B,A)}(x) f,
\end{equation*}
whenever $x \in F_1(B,A), f \in \cat B(B',B), g \in \cat A(A,A')$.
\subsection{Yoneda lemma and Yoneda embedding}
Let $\cat A$ be a dg-category. We associate to $\cat A$ an $\cat A$-$\cat A$-dg-bimodule called the \emph{diagonal bimodule} and denoted by $h_{\cat A} = h$. \nomenclature{$h_{\cat A} = h$}{The diagonal bimodule of a dg-category $\cat A$} It is defined by
\begin{equation}
h_{\cat A}(A,A') = h^A_{A'} = \cat A(A,A'),
\end{equation}
with right and left actions given by composition in $\cat A$. Also, given a dg-functor $F \colon \cat C \to \cat A$, we denote respectively by $h^F$ and $h_F$ the $\cat A$-$\cat C$-dg-bimodule and the $\cat C$-$\cat A$-bimodule defined by:
\begin{equation}
\begin{split}
h^F(C,A) = h^{F(C)}_A = \cat A(F(C),A), \nomenclature{$h^F, h_G, h^F_G$}{Dg-modules induced by dg-functors} \\
h_F(A,C) = h^A_{F(C)} = \cat A(A,F(C)).
\end{split}
\end{equation}
The left and right actions of $\cat A$ and $\cat C$ on $h^F$ are defined by the following compositions in $\cat A$:
\begin{align*}
gf &= g \circ F(f), \\
g'g &= g' \circ g,
\end{align*}
whenever $f \in \cat C(C',C)$, $g \in \cat A(F(C),A)$ and $g' \in \cat A(A,A')$. The actions on $h_F$ are defined analogously. Moreover, if $G \colon \cat B \to \cat A$ is another dg-functor, there is a $\cat B$-$\cat C$-dg-bimodule $h^F_G$ defined by:
\begin{equation}
h^F_G(C,B) = \cat A(F(C),G(B)),
\end{equation} 
with left and right actions defined in a similar way as above.

Taking the components of the diagonal bimodule, we obtain the right dg-modules $h_A = \cat A(-,A)$ and the left dg-modules $h^A = \cat A(A,-)$. A right (resp. left) $\cat A$-dg-module $F$ is said to be \emph{representable} if $F \cong h_A$ (resp. $F \cong h^A$) for some $A \in \cat A$. The well-known Yoneda lemma has a counterpart in the differential graded framework:
\begin{thm}[Dg-Yoneda lemma] \label{thm:yoneda}
Let $F \in \compdg{\cat A}$ be a right $\cat A$-dg-module, and let $A \in \cat A$. Then, there is an isomorphism of complexes:
\begin{equation} \label{eq:yoneda_iso}
\begin{split}
\Natdg(h_A, F) & \isorightarrow F(A), \\
\varphi & \mapsto \varphi_A(1_A),
\end{split}
\end{equation}
natural both in $A$ and $F$.
\end{thm}
From this, we obtain the \emph{(dg) Yoneda embedding}:
\begin{equation} \label{eq:yoneda_embedding}
\begin{split}
h \colon \cat A & \to \compdg{\cat A}, \\
A & \mapsto h_A = \cat A(-,A).
\end{split}
\end{equation}

\subsection{Dg-adjunctions}
The notion of adjoint dg-functors is a direct generalisation of the notion of ordinary adjoint functors.
\begin{defin}
Let $F \colon \cat A \leftrightarrows \cat B \colon G$ be dg-functors. We say that $F$ is a \emph{left adjoint of $G$} (and $G$ is a \emph{right adjoint} of $F$), writing $F \dashv G$, if there is an isomorphism of complexes:
\begin{equation}
\varphi_{A,B} \colon \cat B(F(A),B) \isorightarrow \cat A(A,G(B)),
\end{equation}
natural in both $A$ and $B$.
\end{defin}
As in ordinary category theory, a dg-adjunction $F \dashv G \colon \cat A \leftrightarrows \cat B$ is determined by its \emph{counit} $\varepsilon \colon FG \to 1_{\cat B}$ or by its \emph{unit} $\eta \colon 1_{\cat A} \to GF$; they are both closed and degree $0$ natural transformations, and satisfy the usual universal properties. For instance, for any $f \in \cat A(A, G(B))$ there exists a unique $f' \in \cat B (F(A), B)$ such that $f = G(f')\eta_A$:
\begin{equation} \label{eq:adj_unit}
\begin{gathered}
\xymatrix{
A \ar[d]^{\eta_A} \ar[r]^f & G(B) \\
GF(A) \ar@{.>}[ur]_{G(f')}.
}
\end{gathered}
\end{equation}

Clearly, a dg-adjunction $F \dashv G$ induces adjunctions $Z^0(F) \dashv Z^0(G)$ and $H^0(F) \dashv H^0(G)$, with units and counits naturally induced by the unit and counit of $F \dashv G$.

\subsubsection*{Fully faithful adjoint functors} A dg-functor $F \colon \cat A \to \cat B$ is \emph{(dg-)fully faithful} if for any $A, A' \in \cat A$, the map on hom-complexes
\begin{equation*}
F_{(A,A')} \colon \cat A(A,A') \to \cat B(F(A), F(A'))
\end{equation*}
is an isomorphism of complexes. Moreover, we say that $F$ is \emph{(dg-)essentially surjective} if for any $B \in \cat B$ there exists $A \in \cat A$ such that $B \cong F(A)$ (in other words, $Z^0(F)$ is essentially surjective). When we are given an adjunction $F \dashv G$, then we have the following useful characterisation:
\begin{prop} \label{prop:adj_fullyfaithful}
Let $F \dashv G \colon \cat A \leftrightarrows \cat B$ be an adjunction of dg-functors. Then, $F$ is fully faithful if and only if the unit $\eta \colon 1_{\cat A} \to GF$ is an isomorphism. Dually, $G$ is fully faithful if and only if the counit $\varepsilon \colon FG \to 1_{\cat B}$ is an isomorphism.
\end{prop}
This proposition can be strenghtened, giving a very useful result. It is a direct adaptation of \cite[Lemma 1.3]{johnstone-moerdijk-toposes}, which is proved using the formalism of (co)monads:
\begin{prop} \label{prop:adj_leftinverse_unitiso}
Let $F \dashv G \colon \cat A \leftrightarrows \cat B$ be an adjunction of dg-functors. Then, $GF \cong 1_{\cat A}$ if and only if the unit $\eta \colon 1_{\cat A} \to GF$ is an isomorphism. Dually, $FG \cong 1_{\cat B}$ if and only if the counit $\varepsilon \colon FG \to 1_{\cat B}$ is an isomorphism. In particular, $F$ is fully faithful if and only if $GF \cong 1_{\cat A}$, and $G$ is fully faithful if and only if $FG \cong 1_{\cat B}$.
\end{prop}

%Given a dg-category $\cat A$, a  \emph{right $\cat A$-dg-module} is a dg-functor $\opp{\cat A} \to \compdg{\basering k}$; a \emph{left $\cat A$-dg-module} is a right $\opp{\cat A}$-dg-module, that is, a dg-functor $\cat A \to \compdg{\basering k}$. We set
%\begin{equation*}
%\compdg{\cat A} = \Fundg(\opp{\cat A}, \compdg{\basering k}).
%\end{equation*}
%We have a fully faithful dg-functor
%\begin{equation}
%\begin{split}
%h = h_{\cat A} \colon \cat A & \to \compdg{\cat A}, \\
%A & \mapsto \cat A(-,A),
%\end{split}
%\end{equation}
%which is the dg version of the Yoneda embedding. The \emph{dg-Yoneda lemma} also holds: given $F \in \compdg{\cat A}$, there is a natural isomorphism of complexes
%\begin{equation} \label{eq:dgYoneda}
%F(A) \xrightarrow{\sim} \compdg{\cat A}(h_A, F),
%\end{equation}
%where $h_A = \cat A(-,A)$.
%
%Given dg-categories $\cat A$ and $\cat B$, an \emph{$\cat A$-$\cat B$-dg-bimodule} (covariant in $\cat A$, contravariant in $\cat B$) is a right $\cat B \otimes \opp{\cat A}$-dg-module, namely, a dg-functor $\opp{\cat B} \otimes \cat A \to \compdg{\basering k}$. By convention, the contravariant variable comes first. By \eqref{eq:sec1_iso_fundg_internalhom}, such a bimodule can also be viewed as a dg-functor $\cat A \to \compdg{\cat B}$.

\section{Ends and coends} \label{chapter:dgends}
Let $\cat A$ be a dg-category, and let $F \colon \opp{\cat A} \otimes \cat A \to \compdg{\basering k}$ be a dg-bi(endo)module. The aim is to construct a complex which (co)equalises the right and left actions of $\cat A$ on $F$. This leads to the definition of \emph{(co)end}, given in general for dg-functors $\opp{\cat A} \otimes \cat A \to \cat B$. These notions will give us some very useful computational tools. This section is devoted to the development of ends and coends in dg-category theory; a good readable introduction to (co)end calculus in ordinary category theory can be found in \cite{fosco-coend}. Our treatment is just a particular case of the definitions and results given in enriched category theory: possible references for the general setting are \cite{kelly-enriched} or \cite{dubuc-enriched}. The first results mentioned in this section are all proved with a direct verification of a universal property: for the sake of brevity, we limit ourselves to writing down the statements, leaving the proofs to the reader.
\begin{defin} \label{def:end}
Let $F \colon \opp{\cat A} \otimes \cat A \to \cat B$ be a dg-functor. An \emph{end} of $F$ is an object $X_F \in \cat B$ together with closed degree $0$ maps
\begin{equation*}
\varepsilon_A \colon X_F \to F_A^A
\end{equation*}
for all $A \in \cat A$, satisfying the following universal property:
\begin{equation}
\begin{gathered}
\xymatrix{
X' \ar@/_/[ddr]_{f_{A'}} \ar@/^/[drr]^{f_A} \ar@{.>}[dr]|-{f}\\
&X_F \ar[d]^{\varepsilon_{A'}} \ar[r]_{\varepsilon_A} & F_A^A \ar[d]^{F^A_h} \\
&F^{A'}_{A'} \ar[r]^{F_{A'}^h} & F^A_{A'}}
\end{gathered}
\end{equation}
that is, for any $h \in \cat A(A,A')$ the above square with vertex $X_F$ is commutative, and for any $X'$ together with closed degree $0$ maps $f_A \colon X' \to F_A^A$ such that the ``curved square'' with vertex $X'$ is commutative, there exists a unique closed degree $0$ map $f \colon X' \to X_F$ such that $f_A = \varepsilon_A f$ for all $A \in \cat A$. 
\end{defin}
Dualising, we get the definition of coend:
\begin{defin} \label{def:coend}
Let $F \colon \opp{\cat A} \otimes \cat A \to \cat B$ be a dg-functor. A \emph{coend} of $F$ is an object $Y_F \in \cat B$ together with closed degree $0$ maps
\begin{equation*}
\eta_A \colon F^A_A \to Y_F
\end{equation*}
for all $A \in \cat A$, satisfying the following universal property:
\begin{equation}
\begin{gathered}
\xymatrix{
F^A_{A'} \ar[d]_{F^h_{A'}} \ar[r]^{F^A_h} & F_A^A \ar[d]_{\eta_A} \ar@/^/[ddr]^{g_A} &   \\
F^{A'}_{A'} \ar[r]^{\eta_{A'}} \ar@/_/[drr]_{g_{A'}} & Y_F \ar@{.>}[dr]|-{g} & \\
& & Y',
}
\end{gathered}
\end{equation}
that is, for any $h \in \cat A(A',A)$ the above square with vertex $Y_F$ is commutative, and for any $Y'$ together with closed degree $0$ maps $g_A \colon F_A^A \to Y'$ such that the ``curved square'' with vertex $Y'$ is commutative, there exists a unique closed degree $0$ map $g \colon Y_F \to Y'$ such that $g_A = g \eta_A$ for all $A \in \cat A$. 
\end{defin}
\begin{remark}
Ends and coends are defined as couples $(X_F, (\varepsilon_A))$ or $(Y_F, (\eta_A))$; we will often abuse notation and refer to them as their underlying objects $X_F$ and $Y_F$. 

As for any object defined with a universal property, ends and coends, if they exist, are uniquely determined up to canonical isomorphism, so that we may speak of \emph{the} (co)end of a dg-functor $F \colon \opp{\cat A} \otimes \cat A \to \cat B$. We will adopt the integral notation: the end of $F$ will be denoted by
\begin{equation}
\int_A F(A,A) = \int_A F^A_A, \nomenclature{$\int_A F(A,A)$}{The end of $F$}
\end{equation}
and the coend of $F$ will be denoted by
\begin{equation}
\int^A F(A,A) = \int^A F^A_A. \nomenclature{$\int^A F(A,A)$}{The coend of $F$}
\end{equation}
\end{remark}
The existence of ends and coends is not assured in general; however, it holds true for bimodules, i.e. dg-functors $\opp{\cat A} \otimes \cat A \to \compdg{\basering k}$.
\begin{prop} \label{prop:end_complexes}
Let $F \colon \opp{\cat A} \otimes \cat A \to \compdg{\basering k}$ be an $\cat A$-$\cat A$-bimodule. Then, the end of $F$ is isomorphic to the subcomplex of $\prod_{A \in \cat A} F(A,A)$ defined by
\begin{equation*}
V_F = \{ \varphi = (\varphi_A)_{A \in \cat A}: f \varphi_A = (-1)^{|f||\varphi|} \varphi_{A'} f \quad \forall\,f \in \cat A(A, A') \}.
\end{equation*}
The map $\varepsilon_A \colon V_F \to F(A,A)$ is defined by $\varphi \mapsto \varphi_A$.
\end{prop}
%\begin{proof}
%Let $X$ be a complex, and let $\xi_A \colon X \to F(A,A)$ be chain maps such that the following diagram is commutative:
%\begin{equation*}
%\xymatrix{
%X \ar[r]^{\xi_A} \ar[d]_{\xi_{A'}} & F(A,A) \ar[d]^{F(1_A \otimes f)} \\
%F(A',A') \ar[r]^{F(f \otimes 1_{A'})} & F(A,A'),
%}
%\end{equation*}
%for all $f \colon A \to A'$. This means that, for all $x \in X$,
%\begin{equation*}
%f \xi_A(x) = (-1)^{|x||f|}\xi_{A'}(x)f,
%\end{equation*}
%employing the left-right action notation. We look for a chain map $\xi \colon X \to V_F$ such that
%\begin{equation*}
%\xi(x)_A = \xi_A(x)
%\end{equation*}
%for all $x \in X$. Clearly, this identity uniquely defines $\xi$. Since the $\xi_A$ are chain maps, so is $\xi$; the above commutativity of the $\xi_A$ with the left and right actions ensure that $\xi$ takes values in $V_F$, and we are done.
%\end{proof}
\begin{coroll} \label{coroll:nattrans_end}
Let $F, G \colon \cat A \to \cat B$ be dg-functors. Then, the complex of dg-natural transformation $\Natdg(F,G)$, together with the canonical maps
\begin{align*}
\Natdg(F,G) & \to \cat B(F(A),G(A)), \\
\varphi & \mapsto \varphi_A,
\end{align*}
is an end of the bimodule $h^F_G = \cat B(F(-),G(-))$:
\begin{equation}
\int_A \cat B(F(A), G(A)) \cong \Natdg(F,G).
\end{equation} 
\end{coroll}
\begin{prop}
Let $F \colon \opp{\cat A} \otimes \cat A \to \compdg{\basering k}$ be an $\cat A$-$\cat A$-bimodule. Then, the coend of $F$ is isomorphic to the complex
\begin{align*}
W_F =  \coker \left( \bigoplus_{A_1, A_2 \in \cat A} \cat A(A_2, A_1)  \otimes \right. & F(A_1, A_2)  \left. \longrightarrow \bigoplus_{A \in \cat A} F(A,A) \right) \\
f \otimes x & \mapsto fx - (-1)^{|f||x|}xf,
\end{align*}
together with the natural maps
\begin{equation*}
\eta_A \colon F(A,A) \to \bigoplus_{A'} F(A',A') \to W_F.
\end{equation*}
\end{prop}
%\begin{proof}
%Let $X$ be a complex, and let $\xi_A \colon F(A,A) \to X$ be chain maps such that the following diagram is commutative:
%\begin{equation*}
%\xymatrix{
%F(A,A') \ar[r]^{F^A_f} \ar[d]^{F^f_{A'}} & F(A,A) \ar[d]^{\xi_A} \\
%F(A',A') \ar[r]^{\xi_{A'}} & X,
%}
%\end{equation*}
%for all $f \colon A' \to A$. This means that, for any $x \in F(A,A')$:
%\begin{equation*}
%\xi_A(fx) = (-1)^{|f||x|} \xi_{A'}(xf).
%\end{equation*}
%Hence, we find that the composition
%\begin{equation*}
%\bigoplus_{A_1,A_2} \cat A(A_2, A_1) \otimes F(A_1,A_2) \to \bigoplus_A F(A,A) \xrightarrow{\oplus \xi_A} X
%\end{equation*}
%is zero, and so $\oplus \xi_A$ factors through a unique map $\xi \colon W_F \to X$. This is precisely the required universal property.
%\end{proof}
In the following, we explore the properties of ends and coends, and we develop the tools of (co)end calculus.
\begin{defin} \label{defin:cap1_end_preservation}
Let $F \colon \opp{\cat A} \otimes \cat A \to \cat B$ and $G \colon \cat B \to \cat C$ be dg-functors. Assume that $\left( \int_A F(A,A), (\varepsilon_A) \right)$ is an end of $F$. We say that \emph{$G$ preserves the end $\int_A F(A,A)$}  if $\left( G \left( \int_A F(A,A) \right), (G(\varepsilon_A)) \right)$ is an end of $GF$. Dualising, we directly get the definition of \emph{preservation of coends}.
\end{defin}
We will often allow ourselves to abuse notation and write for instance
\begin{equation*}
G(\int_A F(A,A)) \cong \int_A GF(A,A),
\end{equation*}
to mean that $G$ preserves the end of $F$. The following is an important result:
\begin{prop} \label{prop:cap2_hom_preserves_ends}
The hom dg-functor preserves ends. That is, given a dg-functor $F \colon \opp{\cat A} \otimes \cat A \to \cat B$ and assuming that $\int_A F(A,A)$ and $\int^A F(A,A)$ both exist, then:
\begin{align}
\cat B\left(B, \int_A F(A,A)\right) & \cong \int_A \cat B(B, F(A,A)), \label{eq:hom_end} \\
\cat B \left(\int^A F(A,A), B \right) & \cong \int_A \cat B(F(A,A),B), \label{eq:hom_coend}
\end{align}
for all $B \in \cat B$.
\end{prop}
%\begin{proof}
%We prove \eqref{eq:hom_end}, the other statement being dual. We have to check that, if $\varepsilon_A \colon \int_A F(A,A)\to F(A,A)$ are the canonical maps associated to the end of $F$, then the family of maps
%\begin{equation*}
%h^B_{\varepsilon_A} = \cat B(B, \varepsilon_A) = (\varepsilon_A)_* \colon \cat B \left(B, \int_A F(A,A) \right) \to \cat B(B,F(A,A))
%\end{equation*}
%satisfies the universal property of $\int_A \cat B(B,F(A,A))$:
%\begin{equation*}
%\begin{gathered}
%\xymatrix{
%X \ar@/_/[ddr]_{\xi_{A'}} \ar@/^/[drr]^{\xi_A} \ar@{.>}[dr]|-{\xi}\\
%& \cat B(B, \int_A F(A,A)) \ar[d]^{(\varepsilon_{A'})_*} \ar[r]_{(\varepsilon_A)_*} & \cat B(B,F(A,A)) \ar[d]^{h^B_{F(1_A \otimes f)}} \\
%&\cat B(B,F(A',A')) \ar[r]^{h^B_{F(f \otimes 1_{A'})}} & \cat B(B,F(A,A')).}
%\end{gathered}
%\end{equation*}
%Let $x \in X$. Then, we have the following diagram:
%\begin{equation*}
%\begin{gathered}
%\xymatrix{
%B \ar@/_/[ddr]_{\xi_{A'}(x)} \ar@/^/[drr]^{\xi_A(x)} \ar@{.>}[dr]|-{\xi(x)}\\
%&\int_A F(A,A) \ar[d]^{\varepsilon_{A'}} \ar[r]_{\varepsilon_A} & F(A,A) \ar[d]^{F(1_A \otimes f)} \\
%&F(A',A') \ar[r]^{F(f \otimes 1_{A'})} & F(A,A').}
%\end{gathered}
%\end{equation*}
%So, we find a unique $\xi(x) \in Z^0(\cat B(B,\int_A F(A,A))$ such that $\varepsilon_A \circ \xi(x) = \xi_A(x)$. By this uniqueness property, we easily prove that $x \mapsto \xi(x)$ defines a chain map $X \to \cat B(B, \int_A F(A,A))$. By construction, it satisfies $(\varepsilon_A)_* \circ \xi = \xi_A$, and it is the unique with this property.
%\end{proof}
\begin{remark} \label{remark:strong_univ_prop_ends}
The isomorphisms \eqref{eq:hom_end} and \eqref{eq:hom_coend} are actually a stronger (yet equivalent) version of the universal properties which define ends and coends. Recalling the characterisation of ends of $\compdg{\basering k}$-valued dg-functors of Proposition \ref{prop:end_complexes}, we see for instance that \eqref{eq:hom_end} is equivalent to the following statement: for any family of maps $\xi_A \colon B \to F(A,A)$ of degree $p$ such that the following diagram
\begin{equation*}
\xymatrix{
B \ar[r]^{\xi_A} \ar[d]_{\xi_{A'}} & F(A,A) \ar[d]^{F(1_A \otimes f)} \\
F(A',A') \ar[r]^{F(f \otimes 1_{A'})} & F(A,A')
}
\end{equation*}
is commutative up to the sign $(-1)^{pq}$, for all $f \colon A \to A'$ of degree $q$, there exists a unique $\xi \colon B \to \int_A F(A,A)$ of degree $p$ such that $\varepsilon_A \xi = \xi_A$ for all $A \in \cat A$.
\end{remark}
We are now able to show that ends and coends are dg-functorial. We write down statements only for ends, the case of coends being analogous.
\begin{prop}[Dg-functoriality]
Let $F, G \colon \opp{\cat A} \otimes \cat A \to \cat B$ be dg-functors, and assume that $\int_A F(A,A)$ and $\int_A G(A,A)$ exist. If $\varphi \colon F \to G$ is a dg-natural transformation, then there exists a natural morphism
\begin{equation*}
\int_A \varphi \colon \int_A F(A,A) \to \int_A G(A,A).
\end{equation*}
The mapping $\varphi \mapsto \int_A \varphi$ is a chain map, and moreover we have
\begin{align*}
\int_A \psi \varphi &= \int_A \psi \circ \int_A \varphi, \\
\int_A 1_F &= 1_{\int_A F(A,A)},
\end{align*}
assuming $\psi \colon G \to H$ and the existence of $\int_A H(A,A)$.
\end{prop}
\begin{proof}
Define $\int_A \varphi$ with the strong universal property explained in Remark \ref{remark:strong_univ_prop_ends}:
\begin{equation*}
\begin{gathered}
\xymatrix{
\int_A F(A,A) \ar@/_/[ddr]_{\varphi_{A',A'} \varepsilon_{A'}} \ar@/^/[drr]^{\varphi_{A,A}\varepsilon_A} \ar@{.>}[dr]|-{\int_A \varphi}\\
& \int_A G(A,A) \ar[d]^{\varepsilon_{A'}} \ar[r]_{\varepsilon_A} & G(A,A) \ar[d] \\
& G(A',A') \ar[r] & G(A,A').}
\end{gathered}
\end{equation*}
The properties required follow by uniqueness arguments.
\end{proof}
\begin{remark} \label{remark:cap1_ends_parameters}
Let $F \colon \opp{\cat A} \otimes \cat A \otimes \cat C \to \cat B$ be a dg-functor. Assume that, for all $C \in \cat C$, the end $\int_A F(A,A,C)$ exists. Then, it is dg-functorial in $C$. indeed, given $f \colon C \to C'$, we obtain a dg-natural transformation
\begin{equation*}
\varphi_f = F(-,-,f) \colon F(-,-,C) \to F(-,-,C')
\end{equation*}
of functors $\opp{\cat A} \otimes \cat A \to \cat B$, and by dg-functoriality we get a natural morphism:
\begin{equation*}
\int_A \varphi_f \colon \int_A F(A,A,C) \to \int_A F(A,A,C').
\end{equation*}
The mapping $f \mapsto \varphi_f \mapsto \int_A \varphi_f$ is dg-functorial, so in the end we get a dg-functor
\begin{equation*}
\int_A F(A,A,-) \colon \cat C \to \cat B,
\end{equation*}
together with natural transformations
\begin{equation*}
\varepsilon_A \colon \int_A F(A,A,-) \to F(A,A,-).
\end{equation*}
This is the ``end with parameters''. The same discussion can obviously be done for coends.
\end{remark}
The following result is an ``interchange law'' for (co)ends. With the integral notation, it becomes a ``categorical Fubini theorem''. We give the statement for ends:
\begin{prop}[``Fubini theorem''] \label{prop:Fubini}
Let $F \colon \opp{\cat A} \otimes \opp{\cat B} \otimes \cat A \otimes \cat B \to \cat C$ be a dg-functor. Assume that for all $A, A' \in \cat A$, the end
\begin{equation*}
\int_B F^{A, B}_{A', B}
\end{equation*}
exists. Then, there is a natural isomorphism:
\begin{equation*}
\int_{(A,B)} F^{A,B}_{A,B} \cong \int_A \int_B F^{A,B}_{A,B},
\end{equation*}
whenever one of these two ends exists. Moreover, if for all $B, B' \in \cat B$, the end $\int_A F^{A,B}_{A,B'}$ also exists, then
\begin{equation} \label{eq:Fubini}
\int_{(A,B)} F^{A,B}_{A,B} \cong \int_A \int_B F^{A,B}_{A,B} \cong \int_B \int_A F^{A,B}_{A,B},
\end{equation}
whenever one of these ends exist.
\end{prop}
%\begin{proof}
%Assume that $\int_{(A,B)} F^{A,B}_{A,B}$ exists, and let $\varepsilon_{(A,B)} \colon \int_{(A,B)} F^{A,B}_{A,B} \to F^{A,B}_{A,B}$ be the natural associated maps. Apply the universal property of $\int_B F^{A,B}_{A,B}$ (together with the associated maps $p_B \colon \int_B F^{A,B}_{A,B} \to F^{A,B}_{A,B}$):
%\begin{equation*}
%\xymatrix{
%\int_{(A,B)} F^{A,B}_{A,B} \ar[r]^-{\varepsilon_{(A,B)}} \ar@{.>}[d]^{q_A} & F^{A,B}_{A,B} \\
%\int_B F^{A,B}_{A,B}. \ar[ur]_{p_B}
%}
%\end{equation*}
%So, $(\int_{A,B} F^{A,B}_{A,B}, (q_A))$ satisfies the universal property of $\int_A \int_B F^{A,B}_{A,B}$. Conversely, assume that $(\int_A \int_B F^{A,B}_{A,B}, (q_A))$ is an end of $\int_B F^{-,B}_{-,B}$; then, define $\varepsilon_{(A,B)} \colon \int_A \int_B F^{A,B}_{A,B} \to F^{A,B}_{A,B}$ as $\varepsilon_{(A,B)} = p_B q_A$, and check that they satisfy the universal property of $\int_{(A,B)} F^{A,B}_{A,B}$. We leave the reader to fill in the details, and to conclude with the proof of \eqref{eq:Fubini}.
%\end{proof}
\subsection{Yoneda lemma, revisited}
The complex of natural transformations between two dg-functors can be written as an end, as we have already seen in Corollary \ref{coroll:nattrans_end}. So, it is clear that Yoneda lemma can be restated employing this formalism.
\begin{prop}[Yoneda lemma] \label{prop:yoneda_end}
Let $\cat A$ be a dg-category, let $F \colon \cat A \to \compdg{\basering k}$ and $G \colon \opp{\cat A} \to \compdg{\basering k}$ be respectively a left and a right $\cat A$-dg-module. Then:
\begin{equation} \label{eq:yoneda_end}
\begin{split}
F_- & \cong \int_A \compdg{\basering k}(h^-_A, F_A), \\
G^- & \cong \int_A \compdg{\basering k}(h^A_-, G^A).
\end{split}
\end{equation}
where the natural maps $\varepsilon_A \colon F_- \to \compdg{\basering k}(h_A^-, F_A)$ and $\varepsilon'_A \colon G^- \to \compdg{\basering k}(h^A_-, G^A)$ are defined respectively by
\begin{align*}
\varepsilon_A(x)(f) &= (-1)^{|x||f|} fx, \\
\varepsilon'_A(y)(g) &= yg.
\end{align*}
\end{prop}
%\begin{proof}
%We prove the second statement, the other one being dual. Let $V$ be a complex, let $X \in \cat A$ and let $\xi_A \colon V \to \compdg{\basering k}(h^A_X, G^A)$ be chain maps such that the diagram
%\begin{equation*}
%\xymatrix{
%V \ar[r]^{\xi_A}  \ar[d]^{\xi_{A'}} & \compdg{\basering k}(h^A_X, G^A) \ar[d]^{(G^g)_*} \\
%\compdg{\basering k}(h^{A'}_X, G^{A'}) \ar[r]^{(h^g_X)^*} & \compdg{\basering k}(h^A_X, G^{A'})
%}
%\end{equation*}
%is commutative  for all $g \colon A' \to A$. We want to define $\xi \colon V \to G^X$ such that $\varepsilon'_A \xi = \xi_A$ for all $A \in \cat A$. In particular, we require that
%\begin{equation*}
%\varepsilon'_A(\xi(v))(h) = \xi(v) h = (-1)^{|v||h|}G(h)(\xi(v)) = \xi_A(v)(h),
%\end{equation*}
%for all $h \colon A \to X$. So, we are forced to set
%\begin{equation*}
%\xi(v) = \xi_X(v)(1_X).
%\end{equation*}
%We leave the reader to check that this definition satisfies the required properties (it is a rather tedious but straightforward verification).
%
%Assuming we have already proved Yoneda lemma with the usual direct argument, we can give a shorter proof of the isomorphisms \eqref{eq:yoneda_end}: indeed, again concentrating on the second one, we just have to check that the diagram
%\begin{equation*}
%\xymatrix{
%\Natdg(h_X, G) \ar[r]^-{\sim} \ar[d] & G^X \ar[dl]_{\varepsilon'_A} \\
%\compdg{\basering k}(h^A_X, G^A)
%} \qquad
%\xymatrix{
%\varphi \ar@{|->}[r] \ar@{|->}[d] & \varphi_X(1_X) \ar@{|->}[dl] \\
%\varphi_A
%}
%\end{equation*}
%is commutative.
%\end{proof}
Interestingly, there is a dual version of Yoneda lemma, which involves coends:
\begin{prop}[Co-Yoneda lemma] \label{prop:coYoneda}
Let $\cat A$ be a dg-category, let $F \colon \cat A \to \compdg{\basering k}$ and $G \colon \opp{\cat A} \to \compdg{\basering k}$ be respectively a left and a right $\cat A$-dg-module. Then:
\begin{equation} \label{eq:coYoneda}
\begin{split}
F_- & \cong \int^A h^A_- \otimes F_A, \\
G^- & \cong \int^A G^A \otimes h^-_A,
\end{split}
\end{equation}
where the associated maps $\eta_A \colon h^A_- \otimes F_A \to F_-$ and $\eta'_A \colon G^A \otimes h_A^- \to G^-$ are induced by the (left and right) actions of $\cat A$:
\begin{align*}
\eta_A(f \otimes x) &= fx, \\
\eta'_A(y \otimes g) &= yg.
\end{align*}
\end{prop}
\begin{proof}
We prove only the first isomorphism, the other one being dual. Let $V$ be a complex, and let $X \in \cat A$. We have the following chain of natural isomorphisms:
\begin{align*}
\int_A \compdg{\basering k}(h^A_X  & \otimes F_A, V) \\
& \cong \int_A \compdg{\basering k}(h^A_X, \compdg{\basering k}(F_A, V)) \\
&\overset{\text{(Yon.)}}\cong \compdg{\basering k}(F_X, V).
\end{align*}
This implies that $F_X$ represents the dg-functor $V \mapsto \int_A \compdg{\basering k}(h^A_X  \otimes F_A, V)$, and so by definition (recall the ``strong universal property'', Remark \ref{remark:strong_univ_prop_ends}) it is the expected coend:
\begin{equation*}
F_X \cong \int^A h^A_X \otimes F_A.
\end{equation*}
To understand what are the associated maps $\eta_A$, we follow the above chain backwards, starting from the unit $1_{F_X}$ and keeping track of its image:
\begin{align*}
1_{F_X} &\mapsto (\varepsilon_A(1_X)(f) = F(f)^*(1_X) = F(f)) \\
& \mapsto (\eta_A(f \otimes x) = F(f)(x) = fx). \qedhere
\end{align*}
\end{proof}
The above proof follows a typical pattern in (co)end calculus. To show that a certain object $X$ is a (co)end, we try to prove that it represents the suitable functor, and in doing so we make use of the computational tools developed so far: (co)end preservation, dg-functoriality, Fubini theorem, and so on. Typically, we end up writing a chain of natural isomorphisms. At every step, we should keep track of the natural maps associated to the written (co)end; the isomorphisms of the chain will always preserve them, and this knowledge allows us to understand what are the natural maps associated to the object $X$, as we did in the second part of the above proof.

\section{The derived category}
Let $\cat A$ be a dg-category. The homotopy category $\hocomp{\cat A}$ has a natural triangulated structure, induced objectwise from that of complexes of $\basering k$-modules. Let $M,N$ be right $\cat A$-modules; a morphism $M \to N$ in $\hocomp{\cat A}$ is a \emph{quasi-isomorphism} if $N(A) \to M(A)$ is a quasi-isomorphism of complexes for all $A \in \cat A$. $M$ and $N$ are said to be \emph{quasi-isomorphic} ($M \qis N$) if there exists a zig-zag  of quasi-isomorphisms:
\begin{equation*}
M \leftarrow M_1 \rightarrow M_2 \leftarrow \cdots \rightarrow N.
\end{equation*}
The \emph{derived category} of $\cat A$ is defined as the localisation of $\hocomp{\cat A}$ along quasi-isomorphisms:
\begin{equation*}
\dercomp{\cat A} = \hocomp{\cat A}[\mathrm{Qis}^{-1}]
\end{equation*}
%The category $\dercomp{\cat A}$ is triangulated; as in the case of the derived category of complexes of $\basering k$-modules, morphisms $T \to T'$ in $\dercomp{\cat A}$ are represented by ``roofs''
%\begin{equation*}
%T \xleftarrow{\approx} T'' \rightarrow T',
%\end{equation*}
%in $\hocomp{\cat A}$, where the arrow $T'' \to T$ is a quasi-isomorphism. Two $\cat A$-dg-modules $T$ and $T'$ are \emph{quasi-isomorphic} ($T \qis T'$) if they are isomorphic in $\dercomp{\cat A}$, which is equivalent to saying that there is a ``roof'' of quasi-isomorphisms:
%\begin{equation}
%T \xleftarrow{\approx} T'' \xrightarrow{\approx} T'.
%\end{equation}

When $\cat A = \basering k$, viewing the base ring as a dg-category with a single object, its derived category is by definition the derived category $\dercomp{\basering k}$ of complexes of $\basering k$-modules. The well-known results about $\dercomp{\basering k}$ have a direct generalisation to $\dercomp{\cat A}$ for any $\cat A$. We recollect them in the following statement.
\begin{defin}
Let $\cat A$ be a dg-category. A right $\cat A$-dg-module $M \in \compdg{\cat A}$ is \emph{acyclic} if $M(A)$ is an acyclic complex for all $A \in \cat A$. The full dg-subcategory of $\compdg{\cat A}$ of acyclic modules is denoted by $\acy{\cat A}$. \nomenclature{$\acy{\cat A}$}{The dg-category of acyclic right $\cat A$-modules, for a given dg-category $\cat A$}
\end{defin}
\begin{prop}[{\cite[Lemma 3.3]{keller-dgcat}}]
$\dercomp{\cat A}$ has a natural structure of triangulated category such that the localisation functor $\delta = \delta_{\cat A} \colon \hocomp{\cat A} \to \dercomp{\cat A}$ is exact.

A morphism $\alpha \colon F \to G$ in $\hocomp{\cat A}$ is a quasi-isomorphism if and only if its cone $\cone(\alpha)$ is acyclic. Moreover, $\dercomp{\cat A}$ is the Verdier quotient of $\hocomp{\cat A}$ modulo the acyclic modules:
\begin{equation}
\dercomp{\cat A} \cong \hocomp{\cat A} / \acy{\cat A}.
\end{equation}
\end{prop}

\subsection{Resolutions} The localisation functor $\delta \colon \hocomp{\cat A} \to \dercomp{\cat A}$ can (and will) be assumed to be the identity on objects. Morphisms in $\dercomp{\cat A}$ are represented by ``roofs'' in $\hocomp{\cat A}$:
\begin{equation}
F \xleftarrow{\approx} F' \to G,
\end{equation}
where $F' \xrightarrow{\approx} F$ is a quasi-isomorphism. The idea is that $F'$ is suitable \emph{resolution} of $F$. In particular, we may assume it is a \emph{h-projective resolution}:
\begin{defin}
Let $F \in \compdg{\cat A}$. $F$ is \emph{h-projective} if, for all $N \in \acy{\cat A}$, the complex $\compdg{\cat A}(F, N)$ is acyclic. This equivalent to requiring that
\begin{equation*}
H^0(\compdg{\cat A})(F,N) = \hocomp{\cat A}(F, N) \cong 0
\end{equation*}
for all $N \in \acy{\cat A}$. The full dg-subcategory of $\compdg{\cat A}$ of h-projective dg-modules is denoted by $\hproj{\cat A}$. \nomenclature{$\hproj{\cat A}$}{The dg-category of h-projective right $\cat A$-dg-modules}
\end{defin}
The shift of a h-projective module is again h-projective; the same is true for the cone of a morphism of h-projective modules. so, $H^0(\hproj{\cat A})$ is a triangulated subcategory of $\hocomp{\cat A}$. Moreover, we have the following characterisation:
\begin{prop}[{\cite[Proposition 10.12.2.2]{bernstein-lunts-equivariant}}]
Let $P \in \compdg{\cat A}$ be a right $\cat A$-dg-module. Then, $P \in \hproj{\cat A}$ if and only if
\begin{equation*}
\delta \colon \hocomp{\cat A}(P,M) \to \dercomp{\cat A}(P,M)
\end{equation*}
is an isomorphism for all $M \in \compdg{\cat A}$.
\end{prop}
\begin{remark}
From the above proposition, we see that any quasi-isomorphism between h-projective dg-modules is actually a homotopy equivalence.
\end{remark}
The following result ensures the existence of h-projective resolutions, and explains their features:
\begin{prop}[{\cite[Proposition 3.1]{keller-dgcat}}] \label{prop:hproj_resolutions}
Any dg-module $F$ admits a \emph{h-projective resolution}, that is, a quasi-isomorphism
\begin{equation}
q_F \colon Q(F) \xrightarrow{\approx} F,
\end{equation}
natural in $F \in \hocomp{\cat A}$, where $Q(F)$ is h-projective. Moreover, $Q$ yields a fully faithful left adjoint $Q \colon \dercomp{\cat A} \to \hocomp{\cat A}$ to the localisation functor $\delta \colon \hocomp{\cat A} \to \dercomp{\cat A}$. The adjunction is obtained as follows:
\begin{equation} \label{eq:resolutions_adjunction_def}
\hocomp{\cat A}(Q(M), N) \xrightarrow{\delta} \dercomp{\cat A}(Q(M),N) \xrightarrow{(\delta(q_M)^{-1})^*} \dercomp{\cat A}(M,N).
\end{equation}
\end{prop}
\begin{coroll} \label{coroll:hproj_enh_der}
The functor
\begin{equation} \label{eq:hproj_enh_der}
H^0(\hproj{\cat A}) \hookrightarrow \hocomp{\cat A} \xrightarrow{\delta} \dercomp{\cat A}
\end{equation}
is an equivalence of triangulated categories. 
\end{coroll}
%\begin{lemma} \label{lemma:cap3_hproj_stronglypretr}
%$\hproj{\cat A}$ is a strongly pretriangulated full dg-subcategory of the dg-category $\compdg{\cat A}$.
%\end{lemma}
%\begin{proof}
%It is sufficient to show that, if $P \in \hproj{\cat A}$, then $P[i] \in \hproj{\cat A}$ for all $i \in \mathbb Z$, and that $\cone(f) \in \hproj{\cat A}$ whenever $f \colon P \to P'$ is a closed degree $0$ morphism between h-projective modules. Let $M \in \acy{\cat A}$. Then:
%\begin{equation*}
%\hocomp{\cat A}(P[i],M) \cong \hocomp{\cat A}(P, M[-i]) \cong 0,
%\end{equation*}
%because $M[-i] \in \acy{\cat A}$, so $P[i] \in \hproj{\cat A}$. Moreover, since $\hocomp{\cat A}$ is a triangulated category, there is an exact sequence:
%\begin{equation*}
%\cdots \to \hocomp{\cat A}(P[1], M) \to \hocomp{\cat A}(\cone(f), M) \to \hocomp{\cat A}(P',M) \to \hocomp{\cat A}(P,M) \to \cdots.
%\end{equation*}
%So, since $\hocomp{\cat A}(P,M)$ and $\hocomp{\cat A}(P',M)$ are zero, the same is true for $\hocomp{\cat A}(\cone(f), M)$. This tells us that $\cone(f) \in \hproj{\cat A}$.
%\end{proof}

The above discussion can be dualised. In fact, morphisms $F \to G$ in $\dercomp{\cat A}$ can also be represented as ``coroofs'':
\begin{equation*}
F \xrightarrow{\approx} R(F) \leftarrow G,
\end{equation*}
where $F \to R(F)$ is a \emph{h-injective resolution}. The results discussed above have their obvious counterparts. For the reader's convenience, we state the definition of h-injective dg-module and the analogue of Proposition \ref{prop:hproj_resolutions}:
\begin{defin}
Let $\cat A$ be a dg-category, and let $F \in \compdg{\cat A}$. $F$ is \emph{h-injective} if, for all $N \in \acy{\cat A}$, the complex $\compdg{\cat A}(N, F)$ is acyclic. This equivalent to requiring that
\begin{equation*}
H^0(\compdg{\cat A})(N,F) = \hocomp{\cat A}(N, F) \cong 0
\end{equation*}
for all $N \in \acy{\cat A}$. The full dg-subcategory of $\compdg{\cat A}$ of h-injective dg-modules is denoted by $\hinj{\cat A}$. \nomenclature{$\hinj{\cat A}$}{The dg-category of h-injective $\cat A$-dg-modules}
\end{defin}
\begin{prop}[{\cite[Proposition 3.1]{keller-dgcat}}] \label{prop:hinj_resolutions}
Every dg-module $F$ admits a \emph{h-injective resolution}, that is, a quasi-isomorphism
\begin{equation}
r_F \colon F \xrightarrow{\approx} R(F),
\end{equation}
natural in $F \in \hocomp{\cat A}$, where $R(F)$ is h-injective. Moreover, $Q$ yields a fully faithful right adjoint $R \colon \dercomp{\cat A} \to \hocomp{\cat A}$ to the localisation functor $\delta \colon \hocomp{\cat A} \to \dercomp{\cat A}$. The adjunction is obtained as follows:
\begin{equation} \label{eq:resolutions_adjunction_def_2}
\dercomp{\cat A}(M, N) \xrightarrow{\delta(r_N)_*} \dercomp{\cat A}(M,R(N)) \xrightarrow{\delta^{-1}} \hocomp{\cat A}(M,R(N)).
\end{equation}
\end{prop}
\begin{remark}
If $M$ is an h-projective dg-module, then we may assume without loss of generality that $Q(M)=M$. Analogously, if $N$ is an h-injective dg-module, we may assume that $R(M)=M$.

Moreover, notice that h-projectives (and their resolutions) can be defined also in the opposite category $\opp{\hocomp{\cat A}}$: they coincide with h-injectives (and their resolutions) in $\hocomp{\cat A}$, and vice-versa.
\end{remark}

\subsection{The derived Yoneda embedding}
Let $\cat A$ be a dg-category, and let $A \in \cat A$. Then, the right $\cat A$-module $h_A$ is h-projective. Indeed, let $N \in \acy{\cat A}$ be an acyclic module. Then, by Yoneda lemma:
\begin{equation*}
\hocomp{\cat A}(h_A, N) \cong H^0(N^A) \cong 0.
\end{equation*}
So, we see that the Yoneda embedding $h_{\cat A}$ factors through $\hproj{\cat A}$, yielding
\begin{equation}
h_{\cat A} \colon \cat A \to \hproj{\cat A}.
\end{equation}
Taking $H^0$ and composing with the equivalence $H^0(\hproj{\cat A}) \xrightarrow{\sim} \dercomp{\cat A}$ of Corollary \ref{coroll:hproj_enh_der}, we obtain the so-called \emph{derived Yoneda embedding}:
\begin{equation} \label{eq:der_Yoneda}
H^0(\cat A) \hookrightarrow \dercomp{\cat A}.
\end{equation}
By definition, the essential image of this functor is the category $\qrep{\cat A}$ of quasi-representable right $\cat A$-modules. We denote by $\tria{\cat A}$ the smallest strictly full triangulated subcategory of $\dercomp{\cat A}$ which contains the image of \eqref{eq:der_Yoneda}. Moreover, we denote by $\per{\cat A}$ the idempotent completion of $\tria{\cat A}$, which coincides with the smallest strictly full triangulated subcategory of $\dercomp{\cat A}$ which contains the image of \eqref{eq:der_Yoneda} and it is thick, i.e. closed under direct summands; it can also be characterised as the subcategory of compact objects in $\dercomp{\cat A}$. The derived Yoneda embedding factors through $\tria{\cat A}$:
\begin{equation}
H^0(\cat A) \hookrightarrow \tria{\cat A} \hookrightarrow \per{\cat A}.
\end{equation}
\begin{defin} \label{def:pretr_tr_dgcat}
A dg-category $\cat A$ is \emph{pretriangulated} if $H^0(\cat A) \hookrightarrow \tria{\cat A}$ is an equivalence; it is \emph{triangulated} if $H^0(\cat A) \hookrightarrow \per{\cat A}$ is an equivalence.
\end{defin}
If $\cat A$ is a pretriangulated dg-category, then $H^0(\cat A)$ has a natural structure of triangulated category; furthermore, a dg-functor $F \colon \cat A \to \cat B$ between pretriangulated dg-categories induces an exact functor $H^0(F)$ between triangulated categories. For any dg-category $\cat A$, the dg-category $\compdg{\cat A}$ is pretriangulated.

\subsection{Derived functors and derived adjunctions}
Let $\cat A$ and $\cat B$ be dg-categories, and let $F \colon \hocomp{\cat A} \to \hocomp{\cat B}$ be an exact functor (in most situation, it is induced by a dg-functor). A typical question is the following: does $F$ induce an exact functor $F' \colon \dercomp{\cat A} \to \dercomp{\cat B}$ such that the diagram
\begin{equation*}
\xymatrix{
\hocomp{\cat A} \ar[r]^F \ar[d]^{\delta_{\cat A}} & \hocomp{\cat B} \ar[d]^{\delta_{\cat B}} \\
\dercomp{\cat A} \ar[r]^{F'} & \dercomp{\cat B}
}
\end{equation*}
is commutative? The answer is positive if $F$ preserves acyclic $\cat A$-modules (or, equivalently, quasi-isomorphisms). In this case, the induced functor $F'$ is often identified with $F$ itself.

In many situations, however, our given functor $F \colon \hocomp{\cat A} \to \hocomp{\cat B}$ does not preserve acyclics; nevertheless, it always does when restricted to h-projective (or h-injective) dg-modules:
\begin{lemma} \label{lemma:cap3_preservation_acyclics}
Let $F \colon \hocomp{\cat A} \to \hocomp{\cat B}$ be an exact functor. Then, $F$ maps dg-modules which are both acyclic and h-projective (resp. acyclic and h-injective) to acyclics, or equivalently it preserves quasi-isomorphisms between h-projective (resp. h-injective) dg-modules.
\end{lemma}
\begin{proof}
Cones of morphisms between h-projective or h-injective dg-modules are easily seen to be themselves h-projective or h-injective. So, recalling that quasi-isomorphisms are precisely the morphisms whose cone is acyclic, it is clear that $F$ preserves quasi-isomorphisms between h-projectives (resp. h-injectives) if and only if it maps dg-modules which are both acyclic and h-projective (resp. acyclic and h-injective) to acyclics. Now, let $M \in \hocomp{\cat A}$ be h-projective and acyclic (or h-injective and acyclic). By h-projectivity (or h-injectivity) any quasi-isomorphism $M \to 0$ is actually a homotopy equivalence. So, $F(M) \approx F(0) =0$ in $\hocomp{\cat B}$, in particular it is acyclic.  
\end{proof}
Now, even if our functor $F \colon \hocomp{\cat A} \to \hocomp{\cat B}$ does not pass to the derived categories, it induces the so-called \emph{derived functors}. Abstractly, they are defined as Kan extensions:
\begin{defin}
Let $F \colon \hocomp{\cat A} \to \hocomp{\cat B}$ be a functor. A \emph{(total) left derived functor} $\lder F$ of $F$ \nomenclature{$\lder F$}{The left derived functor of $F$} is a \emph{right} Kan extension of $\delta_{\cat B} \circ F$ along $\delta_{\cat A}$:
\begin{equation*}
\xymatrix{
\ar @{} [dr] |{\Uparrow}
\hocomp{\cat A} \ar[d]^{\delta_{\cat A}} \ar[r]^F & \hocomp{\cat B} \ar[d]^{\delta_{\cat B}} \\
\dercomp{\cat A} \ar@{.>}[r]_{\lder F}
& \dercomp{\cat B}.
}
\end{equation*}
Dually, a \emph{(total) right derived functor} $\rder F$ of $F$ \nomenclature{$\rder F$}{The right derived functor of $F$} is a \emph{left} Kan extension of $\delta_B \circ F$ along $\delta_A$.
\end{defin}
Clearly derived functors, being Kan extensions, are uniquely determined up to isomorphism. The above Lemma \ref{lemma:cap3_preservation_acyclics} ensures that derived functors actually exist in our framework, and the following proposition gives their concrete definitions, which is what we will actually use. Its proof can be found in \cite[Theorem 2.2.8]{riehl-cathpy}, in a more general setting.
\begin{prop} \label{prop:cap3_existence_derivedfun}
Let $\cat A$ and $\cat B$ be dg-categories, and let $F \colon \hocomp{\cat A} \to \hocomp{\cat B}$ be an exact functor. We know that $F$ preserves quasi-isomorphisms between h-projectives; then, $F$ admits a left derived functor, obtained as follows:
\begin{equation}
\begin{split}
\lder F & \colon \dercomp{\cat A} \to \dercomp{\cat B}, \\
\lder F &= \delta_{\cat B} \circ F \circ Q_{\cat A}
\end{split}
\end{equation}
where $Q_{\cat A} \colon \dercomp{\cat A} \to \hocomp{\cat A}$ is the h-projective resolution functor of $\cat A$.

Dually, we know that $F$ preserves quasi-isomorphisms between h-injectives; then, $F$ admits a right derived functor, obtained as follows:
\begin{equation}
\begin{split}
\rder F & \colon \dercomp{\cat A} \to \dercomp{\cat B}, \\
\rder F &= \delta_{\cat B} \circ F \circ R_{\cat A},
\end{split}
\end{equation}
where $R_{\cat A} \colon \dercomp{\cat A} \to \hocomp{\cat A}$ is the h-injective resolution functor of $\cat A$.
\end{prop}
\begin{remark}
We have observed that, if the functor $F \colon \hocomp{\cat A} \to \hocomp{\cat B}$ preserves acyclics, then it directly induces a functor $F \colon \dercomp{\cat A} \to \dercomp{\cat B}$ between the derived categories. In this case, we don't need to derive $F$, indeed we immediately see that
\begin{equation*}
F \cong \lder F \cong \rder F \colon \dercomp{\cat A} \to \dercomp{\cat B}.
\end{equation*}
\end{remark}
We will often encounter adjunctions between categories of dg-modules. As expected, just as functors can be derived, the same is true for adjunctions:
\begin{prop} \label{prop:der_adj}
Let $\cat A$ and $\cat B$ be dg-categories, and let
\begin{equation*}
F \dashv G \colon \hocomp{\cat A} \leftrightarrows \hocomp{\cat B}
\end{equation*}
be an adjunction of exact functors. Then, there is a derived adjunction
\begin{equation}
\lder F \dashv \rder G \colon \dercomp{\cat A} \leftrightarrows \dercomp{\cat B},
\end{equation}
which is obtained composing the three adjunctions $Q_{\cat A} \dashv \delta_A$, $F \dashv G$ and $\delta_{\cat B} \dashv R_{\cat B}$. Namely:
\begin{equation}
\begin{split}
\dercomp{\cat B}(\lder F(M), N) & =\dercomp{\cat B}(\delta_{\cat B} F Q_{\cat A}(M), N)\\
& \cong \hocomp{\cat B}(F Q_{\cat A}(M), R_{\cat B}(N)) \\
& \cong \hocomp{\cat A}(Q_{\cat A}(M), G R_{\cat B}(N)) \\
& \cong \dercomp{\cat A}(M, \delta_{\cat A} G R_{\cat B}(N))  \\
& = \dercomp{\cat A}(M, \rder G(N)).
\end{split}
\end{equation}
\end{prop}

%We conclude the discussion showing a typical example of derived adjunction. Let $F \colon \cat A \to \cat B$. We know that the restriction functor $\res_F \colon \compdg{\cat B} \to \compdg{\cat A}$ has a left adjoint:
%\begin{equation*}
%\Ind_F \dashv \res_F \colon \compdg{\cat A} \leftrightarrows \compdg{\cat B}.
%\end{equation*}
%If $M \in \compdg{\cat B}$ is acyclic, then its restriction $\res_F(M)$ is acyclic. So, $\res_F$ induces an exact functor between the derived categories, which we also call $\res_F$:
%\begin{equation*}
%\res_F \colon \dercomp{\cat B} \to \dercomp{\cat A}.
%\end{equation*}
%The functor $\Ind_F$ can be derived, yielding
%\begin{equation}
%\lder \Ind_F \colon \dercomp{\cat A} \to \dercomp{\cat B}.
%\end{equation}
%The restriction functor doesn't need to be derived, so we directly obtain the adjunction
%\begin{equation*}
%\lder \Ind_F \dashv \res_F.
%\end{equation*}
%
%As a final remark, we point out that $\lder \Ind_F$ preserves representable modules, just as $\Ind_F$ (Proposition \ref{prop:cap2_Ind_preserves_representable}). In fact, since any representable $h_A$ is h-projective (we will prove this in the beginning of the next section), we may assume $Q(h_A) = h_A$, so that in the end we have:
%\begin{equation*}
%\lder \Ind_F(h_A) \qis h_{F(A)}.
%\end{equation*}

\section{Derived Isbell duality}
We study a duality result between dg-modules (and also bimodules) which is a vast generalisation of the duality of vector spaces over a field. It is called \emph{Isbell duality}, after John Isbell (see \cite{wood-totalcat} for a reference). Our notation here follows the one found on the $n$Lab\footnote{\texttt{ncatlab.org/nlab/show/Isbell+duality}}.
\begin{prop}[Isbell duality]
Let $\cat A$ be a dg-category. There is a dg-adjunction
\begin{equation}
\mathcal O \dashv \Spec \colon \compdg{\cat A} \leftrightarrows \opp{\compdg{\opp{\cat A}}}, \nomenclature{$\mathcal O \dashv \Spec$}{Isbell duality}
\end{equation}
where $\mathcal O$ and $\Spec$ are defined as follows:
\begin{align*}
\mathcal O(X)_A &= \compdg{\cat A}(X, h_A), \\
\Spec(M)^A &= \compdg{\opp{\cat A}}(M, h^A).
\end{align*}
\end{prop}
\begin{proof}
We have to prove that there is a natural isomorphism of complexes:
\begin{equation} \label{eq:Isbelldual_naturaliso}
\compdg{\opp{\cat A}}(M,\mathcal O(X)) \cong \compdg{\cat A}(X, \Spec(M)).
\end{equation}
We compute:
\begin{align*}
\compdg{\opp{\cat{A}}}(M, \mathcal O(X)) & \cong \int_A \compdg{\basering k}(M_A, \mathcal O(X)_A) \\
& = \int_A \compdg{\basering k}(M_A, \compdg{\cat A}(X, h_A)) \\
& = \int_A \compdg{\basering k}(M_A, \int_{A'} \compdg{\basering k}(X^{A'},h^{A'}_A))
\end{align*}
\begin{align*}
& \cong \int_A \int_{A'} \compdg{\basering k}(M_A \otimes X^{A'}, h^{A'}_A) \\
& \cong \int_{A'} \int_A \compdg{\basering k}(X^{A'}, \compdg{\basering k}(M_A, h^{A'}_A)) \\
& \cong \int_{A'} \compdg{\basering k}(X^{A'}, \compdg{\opp{\cat A}}(M, h^{A'})) \\
& \cong \compdg{\cat A}(X, \Spec(M)). \qedhere
\end{align*}
\end{proof}
%indeed, if $N \in \compdg{\cat A}$ is h-projective and acyclic, then
%\begin{equation*}
%\mathcal O(N) = \compdg{\cat A}(N, h_-)
%\end{equation*}
%is also acyclic. Indeed
%\begin{align*}
%H^i(\mathcal O(N)) & = H^0(\compdg{\cat A}(N[-i], h_-) \\
%&= \hocomp{\cat A}(N[-i], h_-) \\
%& \cong \dercomp{\cat A}(N[-i], h_-) \cong 0.
%\end{align*}
$\mathcal O$ and $\Spec$ admit derived functors, by Proposition \ref{prop:cap3_existence_derivedfun}. So, we obtain the left derived functor
\begin{equation}
\begin{split}
&\lder \mathcal O  \colon \dercomp{\cat A} \to \opp{\dercomp{\opp{\cat A}}}, \\
&\lder \mathcal O (X) = \mathcal O (Q(X)).
\end{split}
\end{equation}
Analogously, $\Spec$ induces the right derived functor
\begin{equation}
\begin{split}
& \rder \Spec  \colon  \opp{\dercomp{\opp{\cat A}}} \to \dercomp{\cat A} , \\
&\rder \Spec (M) = \Spec(Q(M)).
\end{split}
\end{equation}
Notice that we employed the h-projective resolution even for $\rder \Spec$, because of contravariance. By Proposition \ref{prop:der_adj}, we get the derived adjunction
\begin{equation}
\lder \mathcal O \dashv \rder \Spec \colon \dercomp{\cat A} \to \opp{\dercomp{\opp{\cat A}}},
\end{equation}
which we call \emph{derived Isbell duality}.

An object $X \in \compdg{\cat A}$ is called \emph{Isbell autodual} if the unit $X \to \Spec (\mathcal O(X))$ is an isomorphism. If $X=h_A$ is represented by $A \in \cat A$, then
\begin{align*}
\mathcal O(X) &= \mathcal O(h_A) \\
&= \compdg{\cat A}(h_A, h_-) \\
&\cong \cat A(A,-) = h^A,
\end{align*}
and analogously
\begin{align*}
\Spec(h^A) &= \compdg{\opp{\cat A}}(h^A, h^-) \\
& \cong \cat A(-,A) = h_A.
\end{align*}
In the end, we have isomorphisms
\begin{align*}
h_A &\cong \Spec \mathcal O(h_A), \\
h^A &\cong \mathcal O(\Spec(h^A)),
\end{align*}
natural in $A \in \cat A$. By Proposition \ref{prop:adj_leftinverse_unitiso}, we deduce that $\cat A$-dg-modules of the form $h_A$ are Isbell autodual, and also, more precisely:
\begin{lemma} \label{lemma:repr_Isbell_equiv}
The dg-adjunction $\mathcal O \dashv \Spec$ restricts to an adjoint dg-equivalence
\begin{equation*}
\rep{\cat A} \leftrightarrows \opp{\rep{\opp{\cat A}}}, \nomenclature{$\rep{\cat A}$}{The dg-category of representable right $\cat A$-dg-modules}
\end{equation*}
where $\rep{\cat A}$ denotes the dg-category of representable right $\cat A$-modules.

Analogously, the induced adjunction $H^0(\mathcal O) \dashv H^0(\Spec)$ restricts to an adjoint equivalence
\begin{equation*}
\horep{\cat A} \leftrightarrows \opp{\horep{\opp{\cat A}}}, \nomenclature{$\horep{\cat A}$}{The category of homotopy representable $\cat A$-dg-modules}
\end{equation*}
where $\horep{\cat A}$ denotes the full subcategory of $\hocomp{\cat A}$ of $\cat A$-modules $X$ such that $X \approx h_A$ for some $A \in \cat A$.
\end{lemma}
With a little more work, we are able to establish a similar result for the derived adjunction $\lder \mathcal O \dashv \rder \Spec$:
\begin{prop} \label{prop:derived_repr_Isbell_equiv}
The adjunction $\lder \mathcal O \dashv \rder \Spec$ restricts to an adjoint equivalence
\begin{equation*}
\qrep{\cat A} \leftrightarrows \opp{\qrep{\opp{\cat A}}}, \nomenclature{$\qrep{\cat A}$}{The category of quasi-representable $\cat A$-dg-modules}
\end{equation*}
where $\qrep{\cat A}$ is the full subcategory of $\dercomp{\cat A}$ of \emph{quasi-representable $\cat A$-modules}: $X \in \qrep{\cat A}$ if and only if $X$ is quasi-isomorphic to $h_A$ for some $A \in \cat A$.
\end{prop}
\begin{proof}
Let $A \in \cat A$. Then:
\begin{align*}
\lder \mathcal O(h_A) &= \mathcal O(Q(h_A)) \\
&= \compdg{\cat A}(Q(h_A),h_-) \\
&\qis \compdg{\cat A}(h_A, h_-) \\
&\cong h^A,
\end{align*}
and analogously $\rder \Spec(h^A) \qis h_A$. The quasi-isomorphism $\compdg{\cat A}(Q(h_A),h_-) \qis \compdg{\cat A}(h_A, h_-)$ induced by $q \colon Q(h_A) \to h_A$ comes from the fact that both $Q(h_A)$ and $h_A$ are h-projective; it is actually a homotopy equivalence. Since $q$ is natural in $\hocomp{\cat A}$, we deduce that $\lder \mathcal O(h_A) \qis h^A$ and $\rder \Spec(h^A) \qis h_A$ are natural in $A \in H^0(\cat A) \hookrightarrow \dercomp{\cat A}$. So, we have natural isomorphisms
\begin{align*}
h_A & \qis \rder \Spec(\lder \mathcal O(h_A)), \\
h^A & \qis \lder \mathcal O(\rder \Spec(h^A)),
\end{align*}
and since $\qrep{\cat A}$ is the isomorphism closure of the image of $H^0(\cat A)$ in $\dercomp{\cat A}$, we conclude with the desired claim.
\end{proof}
\subsection{Duality for bimodules} \label{subsec:duality_bimod}
(Derived) Isbell duality extends to bimodules. First, let us introduce some notation:
\begin{align*}
\compdg{\cat A, \cat B} &= \compdg{\cat B \otimes \opp{\cat A}}, \nomenclature{$\compdg{\cat A, \cat B}$}{Shorthand for $\compdg{\cat B \otimes \opp{\cat A}}$} \\
\hocomp{\cat A, \cat B} &= \hocomp{\cat B \otimes \opp{\cat A}}, \nomenclature{$\hocomp{\cat A, \cat B}$}{Shorthand for $\hocomp{\cat B \otimes \opp{\cat A}}$} \\
\dercomp{\cat A, \cat B} &= \dercomp{\cat B \lotimes \opp{\cat A}}. \nomenclature{$\dercomp{\cat A, \cat B}$}{Shorthand for $\dercomp{\cat B \lotimes \opp{\cat A}}$}
\end{align*}
$\lotimes$ denotes the \emph{derived tensor product}, which is defined (up to quasi-equivalence) as:
\begin{equation}
\cat A \lotimes \cat B = \cat A^{\mathrm{hp}} \otimes \cat B \qe \cat A \otimes \cat B^{\mathrm{hp}} \qe \cat A^{\mathrm{hp}} \otimes \cat B^{\mathrm{hp}}, \nomenclature{$\cat A \lotimes \cat B$}{The derived tensor product of dg-categories $\cat A$ and $\cat B$}
\end{equation}
where $\cat A^{\mathrm{hp}}$ is a \emph{h-projective resolution} of $\cat A$ (see, for instance, \cite[Remark 2.7]{canonaco-stellari-internalhoms}).These definitions are justified by the observation that a dg-bimodule $F \in \compdg{\cat A, \cat B}$ (covariant in $\cat A$, contravariant in $\cat B$) can be seen as a dg-functor $F \colon \cat A \to \compdg{\cat B}$.

Isbell duality generalises quite directly to the following:
\begin{prop} \label{prop:cap3_duality_bimodules}
Let $\cat A, \cat B$ be dg-categories. There is a dg-adjunction
\begin{equation}
\ldual \dashv \rdual \colon \compdg{\cat A, \cat B} \leftrightarrows \opp{\compdg{\cat B, \cat A}}, \nomenclature{$\ldual \dashv \rdual$}{Duality for bimodules}
\end{equation}
where $\ldual$ and $\rdual$ are defined by
\begin{equation}
\begin{split}
\ldual(T)^A_B &= \mathcal O(T_A)_B = \compdg{\cat B}(T_A, h_B), \\
\rdual(S)^B_A &= \Spec(S^A)^B = \compdg{\opp{\cat B}}(S^A, h^B).
\end{split}
\end{equation}
\begin{proof}
We have to prove that there is a natural isomorphism of complexes:
\begin{equation}
\compdg{\cat B, \cat A}(S, \ldual(T)) \cong \compdg{\cat A, \cat B}(T, \rdual(S)).
\end{equation}
We compute:
\begin{align*}
\compdg{\cat B, \cat A}(S, \ldual(T)) & \cong \int_A \compdg{\opp{\cat B}}(S^A, \mathcal O(T_A)) \\
& \cong \int_A \compdg{\opp{\cat B}}(T_A, \Spec(S^A)) \\
& \cong \compdg{\cat A, \cat B}(T, \rdual(S)),
\end{align*}
where the second isomorphism of the chain follows from the Isbell duality isomorphism \eqref{eq:Isbelldual_naturaliso} of $\cat B$.
\end{proof}
\end{prop}
By Proposition \ref{prop:cap3_existence_derivedfun}, $\ldual$ and $\rdual$ can be derived, and in the end we obtain the derived adjunction:
\begin{equation}
\begin{split}
\lder \ldual & \dashv \rder \rdual \colon \dercomp{\cat A, \cat B}  \leftrightarrows \opp{\dercomp{\cat B, \cat A}}, \\
&\lder \ldual(T) = \ldual(Q(T)), \\
&\rder \rdual(S) = \rdual(Q(S)).
\end{split}
\end{equation}

The above definitions employ h-projective resolutions of bimodules. A bimodule $T \in \compdg{\cat A, \cat B}$ induces right $\cat B$-modules $T_A$ and left $\cat A$-modules $T^B$, for all $A \in \cat A$ and $B \in \cat B$. A very useful result is that an h-projective resolution of $T$ induces componentwise h-projective resolutions of $T_A$ and $T^B$ (for all $A$ and all $B$), as explained in the following lemma:
\begin{lemma} \label{lemma:hproj_component}
Let $\cat A, \cat B$ be h-projective dg-categories. Let $T \in \compdg{\cat A, \cat B}$ be an h-projective bimodule. Then, for all $A \in \cat A$, $T_A \in \compdg{\cat B}$ is h-projective. Analogously, for all $B \in \cat B$, $T^B \in \compdg{\opp{\cat A}}$ is h-projective. In particular, if $q \colon Q(T) \to T$ is an h-projective resolution of $T$, then $q_A \colon Q(T)_A \to T_A$ and $q^B \colon Q(T)^B \to T^B$ are h-projective resolutions respectively of $T_A$ and $T^B$, for all $A \in \cat A$ and $B \in \cat B$. Without loss of generality, we may set $Q(T)_A = Q(T_A)$ and $Q(T)^B = Q(T^B)$.
\end{lemma}

The adjunction $\ldual \dashv \rdual$ and its derived version $\lder \ldual \dashv \rder \rdual$ are strictly related to the (derived) Isbell duality adjunction. indeed, we have the following:
\begin{lemma} \label{lemma:bimodules_duality_units_Isbell}
Let $\cat A, \cat B$ be dg-categories. Let $T \in \compdg{\cat A, \cat B}$ and $S \in \compdg{\cat B, \cat A}$. Let $\eta \colon T \to \rdual \ldual(T)$ and $\varepsilon \colon S \to \ldual \rdual(S)$\footnote{We view the counit as a map in $\compdg{\cat B, \cat A}$: this explains the seemingly ``wrong direction'' of the arrow.} be the unit and counit morphisms of the adjunction $\ldual \vdash \rdual$, calculated in $T$ and $S$. Then, for all $A \in \cat A$, the morphisms $\eta_A \colon T_A \to \rdual \ldual(T)_A$ and $\varepsilon_A \colon S^A \to \ldual \rdual(S)^A$ are the unit and counit maps of the Isbell duality of $\cat B$, calculated in $T_A$ and $S^A$.
\end{lemma}
\begin{proof}
We rewrite the adjunction $\ldual \dashv \rdual$ as follows:
\begin{equation*}
\int_A \compdg{\opp{\cat B}}(S^A, \ldual(T)^A) \isorightarrow \int_A \compdg{\cat B}(T_A, \rdual(S)_A).
\end{equation*}
By definition, $\ldual(T)^A = \mathcal O(T_A)$, $\rdual(S)_A = \Spec(S^A)$, and there is a commutative diagram for all $A \in \cat A$:
\begin{equation*}
\xymatrix{
\displaystyle{\int_A \compdg{\opp{\cat B}}(S^A, \ldual(T)^A)} \ar[r]^{\sim} \ar[d] & \displaystyle{\int_A \compdg{\cat B}(T_A, \rdual(S)_A)} \ar[d] \\
\compdg{\opp{\cat B}}(S^A, \ldual(T)^A) \ar[r]^{\sim} & \compdg{\cat B}(T_A, \rdual(S)_A).
}
\end{equation*}
The vertical arrows are the natural maps associated to the written ends; the ``downstairs'' isomorphism is precisely the Isbell duality adjunction of $\cat B$, and our claim immediately follows.
\end{proof}
The above result immediately extends to the homotopy level adjunction $H^0(\ldual) \dashv H^0(\rdual)$, and also to the derived adjunction $\lder \ldual \dashv \rder \rdual$:
\begin{coroll} \label{coroll:derived_bimodules_duality_units_Isbell}
Let $\cat A, \cat B$ be dg-categories, and let $T \in \dercomp{\cat A, \cat B}, S \in \dercomp{\cat B, \cat A}$. Let $\widetilde{\eta} \colon T \to \rder \rdual(\lder \ldual(T))$ and $\widetilde{\varepsilon} \colon S \to \lder \ldual (\rder \rdual(S))$ be the unit and counit morphisms of the derived adjunction $\lder \ldual \dashv \rder \rdual$, calculated in $T$ and $S$. Then, for all $A \in \cat A$, the morphisms $\widetilde{\eta}_A \colon T_A \to \rder \rdual(\lder \ldual(T))_A$ and $\widetilde{\varepsilon}_A \colon S^A \to \lder \ldual (\rder \rdual(S))^A$ are the unit and counit morphisms of the derived Isbell duality of $\cat B$, calculated in $T_A$ and $S^A$.
\end{coroll}
\begin{proof}
For simplicity, assume that $\cat A$ and $\cat B$ are h-projective, identifying them with their h-projective resolutions. Let $A \in \cat A$. There is an obvious dg-functor
\begin{align*}
(-)_A \colon \compdg{\cat A, \cat B} &\to \compdg{\cat B}, \\
T & \mapsto T_A.
\end{align*}
This dg-functor clearly preserves acyclic modules, hence it induces a functor
\begin{equation*}
(-)_A \colon \dercomp{\cat A, \cat B} \to \dercomp{\cat B}.
\end{equation*}
Recall that, by Lemma \ref{lemma:hproj_component}, if $q \colon Q(T) \to T$ is an h-projective resolution of $T$, then $q_A \colon Q(T)_A = Q(T_A) \to T_A$ is an h-projective resolution of $T_A$. The functor $Q$ is left adjoint to the localisation functor; recalling how this adjunction is obtained (formula \eqref{eq:resolutions_adjunction_def}), we see that the diagram
\begin{equation*}
\xymatrix{
\hocomp{\cat A, \cat B}(Q(T),T') \ar[r]^-{\sim} \ar[d]^{(-)_A} & \dercomp{\cat A, \cat B}(T,T') \ar[d]^{(-)_A} \\
\hocomp{\cat B}(Q(T_A), T'_A) \ar[r]^-{\sim} & \dercomp{\cat B}(T_A, T'_A).
}
\end{equation*}
is commutative. This, combined with the above lemma and with the definition of the adjunction $\lder \ldual \dashv \rder \rdual$ as a composition of adjunctions (Proposition \ref{prop:der_adj}), gives us the claim regarding the unit $\widetilde{\eta}$. A similar argument gives the other part of the statement.
\end{proof}
Now, let $T \in \compdg{\cat A, \cat B}$ be a \emph{right representable} bimodule, that is, for all $A \in \cat A$, $T_A \cong h_{F(A)}$ for some $F(A) \in \cat B$. Then, we have that
\begin{align*}
\ldual (T)^A &= \mathcal O(T_A) \\
&\cong \mathcal O(h_{F(A)}) \\
&\cong h^{F(A)}.
\end{align*}
So, $\ldual(T)$ is \emph{left representable}. Analogously, if $S \in \compdg{\cat B, \cat A}$ is left representable, that is, $S^A \cong h^{G(A)}$ for all $A \in \cat A$, then $\rdual(S)$ is right representable, and in particular $\rdual(S)_A \cong h_{G(A)}$ for all $A$. So, the duality $\ldual \dashv \rdual$ sends right representables to left representables, and vice-versa. A similar observation can be done at the homotopy level: call a bimodule $T \in \compdg{\cat A, \cat B}$ \emph{right homotopy representable} if $T_A \approx h_{F(A)}$ for some $F(A) \in \cat B$, for alla $A \in \cat A$. Then, a similar computation as above shows that $\ldual(T)^A \approx h^{F(A)}$, so that $\ldual(T)$ is \emph{left homotopy representable}. Vice-versa, if $S \in \compdg{\cat B, \cat A}$ is left homotopy representable, then $\rdual(S)$ is right homotopy representable. More precisely, we have the following:
\begin{lemma}
Let $\cat A, \cat B$ be dg-categories. The dg-adjunction $\ldual \dashv \rdual$ restricts to an adjoint dg-equivalence
\begin{equation*}
\rrep{\cat A, \cat B} \leftrightarrows \opp{\lrep{\cat B, \cat A}},
\end{equation*}
where $\rrep{\cat A, \cat B}$ \nomenclature{$\rrep{\cat A, \cat B}$}{The dg-category of right representable $\cat A$-$\cat B$-dg-bimodules} is the full dg-subcategory of right representable bimodules in $\compdg{\cat A, \cat B}$, and $\lrep{\cat B, \cat A}$ \nomenclature{$\lrep{\cat A, \cat B}$}{The dg-category of left representable $\cat A$-$\cat B$-dg-bimodules} is the full dg-subcategory of left representable bimodules in $\compdg{\cat B, \cat A}$.

Analogously, the homotopy adjunction $H^0(\ldual) \dashv H^0(\rdual)$ restricts to an adjoint equivalence
\begin{equation*}
\rhorep{\cat A, \cat B} \leftrightarrows \opp{\lhorep{\cat B, \cat A}},
\end{equation*}
where $\rhorep{\cat A, \cat B}$ \nomenclature{$\rhorep{\cat A, \cat B}$}{The category of right homotopy representable $\cat A$-$\cat B$-dg-bimodules} and $\lhorep{\cat B, \cat A}$ \nomenclature{$\lhorep{\cat A, \cat B}$}{The category of left homotopy representable $\cat A$-$\cat B$-dg-bimodules} denote respectively the full subcategories of $\hocomp{\cat A, \cat B}$ and $\hocomp{\cat B, \cat A}$ of \emph{right (or left) homotopy representable bimodules}.
\end{lemma}
\begin{proof}
This is a direct application of Lemma \ref{lemma:bimodules_duality_units_Isbell}. For instance, to show that the unit $\eta \colon T \to \rdual \ldual(T)$ is an isomorphism when $T \in \rrep{\cat A, \cat B}$, or a homotopy equivalence when $T \in \rhorep{\cat A, \cat B}$, it is sufficient to show that the components $\eta_A \colon T_A \to \rdual \ldual(T)_A$ are such for all $A \in \cat A$. But by hypothesis $T_A \in \rep{\cat B}$ (or $\horep{\cat B}$ in the case of homotopy right representability), so by Lemma \ref{lemma:repr_Isbell_equiv} we are done.
\end{proof}
A similar result as above holds for the derived duality $\lder \ldual \dashv \rder \rdual$. Call a bimodule $T \in \compdg{\cat A, \cat B}$ \emph{right quasi-representable} if $T_A \qis h_{F(A)}$ for some $F(A) \in \cat B$, for all $A \in \cat A$; analogously, a bimodule $S \in \compdg{\cat B, \cat A}$ is called \emph{left quasi-representable} if $S^A \qis h^{G(A)}$ for some $G(A) \in \cat B$, for all $A \in \cat A$. We have the following:
\begin{prop} \label{prop:lqrep_rqrep}
Let $\cat A, \cat B$ be dg-categories. The derived adjunction $\lder \ldual \dashv \rder \rdual$ restricts to an adjoint equivalence
\begin{equation*}
\rqrep{\cat A, \cat B} \leftrightarrows \opp{\lqrep{\cat B, \cat A}},
\end{equation*}
where $\rqrep{\cat A, \cat B}$ \nomenclature{$\rqrep{\cat A, \cat B}$}{The category of right quasi-representable $\cat A$-$\cat B$-dg-bimodules} is the full subcategory of $\dercomp{\cat A, \cat B}$ of right quasi-representable bimodules, and $\lqrep{\cat B, \cat A}$ \nomenclature{$\lqrep{\cat A, \cat B}$}{The category of left quasi-representable $\cat A$-$\cat B$-dg-bimodules} is the full subcategory of $\dercomp{\cat B, \cat A}$ of left quasi-representable bimodules.
\end{prop}
\begin{proof}
This is an application of Corollary \ref{coroll:derived_bimodules_duality_units_Isbell}. For instance, to show that the unit $\widetilde{\eta} \colon T \to \rder \rdual(\lder \ldual(T))$ is an isomorphism in $\dercomp{\cat A, \cat B}$, it is sufficient to show that $\widetilde{\eta}_A$ is an isomorphism in $\dercomp{\cat B}$ for all $A$. This follows directly by Proposition \ref{prop:derived_repr_Isbell_equiv}, since by hypothesis $T_A \in \qrep{\cat B}$ for all $A \in \cat A$.
\end{proof}
\section{The bicategory of bimodules; adjoints} \label{section:bicat_bimod_adj}
An interesting feature of bimodules is that they can be viewed as ``generalised functors''. We will sometimes write $F \colon \cat A \profto \cat B$ \nomenclature{$F \colon \cat A \profto \cat B$}{Notation which means that $F$ is a $\cat A$-$\cat B$-dg-bimodule, or $F \in \compdg{\cat A, \cat B}$} meaning $F \in \compdg{\cat A, \cat B}$. Given bimodules $F \colon \cat A \profto \cat B$ and $G \colon \cat B \profto \cat C$, we can define their composition $G \diamond F \colon \cat A \profto \cat C$, as follows:
\begin{equation}
(G \diamond F)_A^C  = \int^B F^B_A \otimes G^C_B. \nomenclature{$G \diamond F$}{Composition of dg-bimodules}
\end{equation}
Applying the dg-functoriality of coends, we find out that $\diamond$ is dg-functorial in both variables, hence giving rise to dg-bifunctors
\begin{equation}
- \diamond - \colon \compdg{\cat B, \cat C} \otimes \compdg{\cat A, \cat B} \to \compdg{\cat A, \cat C}.
\end{equation}
In particular, if $\varphi \colon F \to F'$ and $\psi \colon G \to G'$ are dg-natural transformations, we have dg-natural transformations $\psi \diamond F \colon G \diamond F  \to G' \diamond F$ and $G \diamond \varphi \colon G \diamond F  \to G \diamond F'$.

By co-Yoneda lemma, the diagonal bimodules act as (weak) units for this composition:
\begin{equation}
\begin{split}
F \diamond h_{\cat A} &= \int^A h^A \otimes F_A \cong F, \\
h_{\cat B} \diamond F &= \int^B F^B \otimes h_B \cong F,
\end{split}
\end{equation}
given $F \colon \cat A \profto \cat B$. Moreover, the composition is weakly associative. indeed, given $F \colon \cat A \to \cat B, G \colon \cat B \to \cat C, H \colon \cat C \to \cat D$, we have:
\begin{align*}
H \diamond (G \diamond F) &= \int^C (G \diamond F)^C \otimes H_C \\
&= \int^C \left(\int^B F^B \otimes G^C_B \right) \otimes H_C \\
&\cong \int^B \int^C F^B \otimes (G^C_B \otimes H_C) \\
&\cong \int^B F^B \otimes \left(\int^C G^C_B \otimes H_C \right) \\
&= \int^B F^B \otimes (H \diamond G)_B \\
& = (H \diamond G) \diamond F,
\end{align*}
where we used Fubini's theorem and the fact that the tensor product preserves coends (exercise).

Another interesting property of the composition $\diamond$ is that it preserves h-projective bimodules:
\begin{lemma} \label{lemma:diamondprod_hproj}
Let $\cat A, \cat B, \cat C$ be h-projective dg-categories. Let $F \colon \cat A \profto \cat B$ and $G \colon \cat B \profto \cat C$ be h-projective bimodules. Then, $G \diamond F$ is h-projective.
\end{lemma}
\begin{proof}
Let $N \in \compdg{\cat A, \cat C}$ be acyclic. We compute:
\begin{align*}
\compdg{\cat A, \cat C}(G \diamond F, N) &= \int_{A,C} \compdg{\basering k}((G \diamond F)^C_A, N^C_A) \\
& =\int_{A,C} \compdg{\basering k}\left(\int^B F^B_A \otimes G^C_B, N^C_A\right) \\
& \cong \int_{A,B,C} \compdg{\basering k}(F^B_A \otimes G_B^C, N_A^C) \\
& \cong \int_{A,B,C} \compdg{\basering k}(F^B_A, \compdg{\basering k}(G^C_B, N^C_A)) \\
& \cong \int_{A,B} \compdg{\basering k}(F^B_A, \compdg{\cat C}(G_B, N_A) \\
& \cong \compdg{\cat A, \cat B}(F, \compdg{\cat C}(G_-, N_-)).
\end{align*}
Now, $(A,B) \mapsto \compdg{\cat C}(G_B, N_A)$ is acyclic, indeed:
\begin{align*}
H^i(\compdg{\cat C}(G_B, N_A)) &= H^0(\compdg{\cat C}(G_B, N_A[i]) \\
&= \hocomp{\cat C}(G_B, N_A[i]),
\end{align*}
and $G_B$ is h-projective by Lemma \ref{lemma:hproj_component}, so $\hocomp{\cat C}(G_B, N_A[i]) \cong 0$. Hence, since $F$ is h-projective, $\compdg{\cat A, \cat B}(F, \compdg{\cat C}(G_-, N_-))$ is acyclic, and we are done.
\end{proof}

There is a derived version of the composition $\diamond$. Namely, given $F \in \dercomp{\cat A, \cat B}$ and $G \in \dercomp{\cat B, \cat C}$, we set
\begin{equation}
G \ldiamond F = Q(G) \diamond Q(F) \qis G \diamond Q(F) \qis Q(F) \diamond G
\end{equation}
taking h-projective resolutions of either $F$ or $G$. The composition $\ldiamond$ is defined up to quasi-isomorphism, and it is functorial, as we can expect, in the sense that it gives bifunctors
\begin{equation}
- \ldiamond - \colon \dercomp{\cat B, \cat C} \otimes \dercomp{\cat A, \cat B} \to \dercomp{\cat A, \cat C}.
\end{equation}
By the above Lemma \ref{lemma:diamondprod_hproj}, $Q(G) \diamond Q(F)$ is always h-projective, so we may prove directly that $\ldiamond$ is weakly associative and unital. indeed:
\begin{align*}
H \ldiamond (G \ldiamond F) &= Q(H) \diamond (G \ldiamond F) \\
&= Q(H) \diamond (Q(G) \diamond Q(F)) \\
& \cong (Q(H) \diamond Q(G)) \diamond Q(F) \\
&= (H \ldiamond G) \diamond Q(F) \\
&= (H \ldiamond G) \ldiamond F,
\end{align*}
and
\begin{align*}
F \ldiamond h_{\cat A} &= Q(F) \diamond h_{\cat A} \cong Q(F) \qis F, \\
h_{\cat B} \ldiamond F &= h_{\cat B} \diamond Q(F) \cong Q(F) \qis F.
\end{align*}
The composition $\diamond$ and its derived version $\ldiamond$ are indeed part of \emph{bicategorical} structures. Namely, we have the (dg-)bicategory $\kat{Bimod}$ \nomenclature{$\kat{Bimod}$}{The (dg-)bicategory of dg-categories and bimodules} whose objects are dg-categories, with $1$-morphisms and $2$-morphisms respectively given by the objects and the morphisms of the dg-categories $\compdg{\cat A, \cat B}$; in the derived setting, we have the bicategory $\kat{DBimod}$ \nomenclature{$\kat{DBimod}$}{The bicategory of dg-categories and derived bimodules} whose objects are dg-categories, with $1$-morphisms and $2$-morphisms given respectively by the objects and the morphisms of the categories $\dercomp{\cat A, \cat B}$. We won't study these structures in full detail; however, it is interesting to explore the notion of adjointness and its relation to (quasi)-representability.
\begin{defin}
Let $F \colon \cat A \profto \cat B$ and $G \colon \cat B \profto \cat A$ be $1$-morphisms in $\kat{Bimod}$. We say that $F \dashv G$ ($F$ is left adjoint to $G$) if there exist (closed, degree $0$) $2$-morphisms $\eta \colon h_{\cat A} \to G \diamond F$ and $\varepsilon \colon F \diamond G \to h_{\cat B}$ such that the following triangular identities are satisfied:
\begin{align*}
& (F \cong F \diamond h_{\cat A} \xrightarrow{F \diamond \eta} F \diamond (G \diamond F) \cong (F \diamond G) \diamond F \xrightarrow{\varepsilon \diamond F} h_{\cat B} \diamond F \cong F) = 1_F, \\
& (G \cong h_{\cat A} \diamond G \xrightarrow{\eta \diamond G} (G \diamond F) \diamond G \cong G \diamond (F \diamond G) \xrightarrow{G \diamond \varepsilon} G \diamond h_{\cat B} \cong G) = 1_G.
\end{align*}

The definition of adjoint $1$-morphisms in $\kat{DBimod}$ is analogous (replace $\diamond$ with the derived composition $\ldiamond$).
\end{defin}
Given $T \in \compdg{\cat A, \cat B}$, we could expect that its dual $\ldual(T)$ (or $\rdual(T)$) were adjoint to $T$. This is not true in general, but it is provable under the right (or the left) representability assumption. First, notice that there is a (closed, degree $0$) morphism in $\compdg{\cat A, \cat A}$, dg-natural in $T$:
\begin{equation} \label{eq:morph_quasiunit}
n \colon \ldual(T) \diamond T \to \compdg{\cat B}(T_-, T_-).
\end{equation}
Indeed, write:
\begin{align*}
(\ldual(T) \diamond T)^A_{A'} &= \int^B T^B_{A'} \otimes \ldual(T)^A_B \\
&= \int^B T^B_{A'} \otimes \compdg{\cat B}(T_A, h_B) \\
& \cong \int^B \compdg{\cat B}(h_B, T_{A'}) \otimes \compdg{\cat B}(T_A, h_B).
\end{align*}
It is directly checked that the diagram
\begin{equation*}
\xymatrix{
\compdg{\cat B}(h_{B'}, T_{A'}) \otimes \compdg{\cat B}(T_A, h_B) \ar[r] \ar[d] & \compdg{\cat B}(h_B, T_{A'}) \otimes \compdg{\cat B}(T_A, h_B) \ar[d] \\
\compdg{\cat B}(h_{B'}, T_{A'}) \otimes \compdg{\cat B}(T_A, h_{B'}) \ar[r] & \compdg{\cat B}(T_A, T_{A'})
}
\end{equation*}
is commutative for all $B \to B'$ in $\cat B$, where the arrows arriving in $\compdg{\cat B}(T_A, T_{A'})$ are given by composition, and they are natural in $A, A'$. Hence, by the universal property of the coend, we find our desired map. There are also (closed, degree $0$) maps
\begin{align}
e \colon T \diamond \compdg{\cat B}(T_-, T_-) & \to T, \label{eq:map_T_diamond} \\
e' \colon \compdg{\cat B}(T_-, T_-) \diamond \ldual(T) &\to \ldual(T). \label{eq:map_diamond_L(T)}
\end{align}
The morphism \eqref{eq:map_T_diamond} is induced by the maps
\begin{align*}
\compdg{\cat B}(T_A, T_{A'}) \otimes T^B_A &\to T^B_{A'}, \\
\varphi \otimes x & \mapsto \varphi^B(x),
\end{align*}
natural in $A'$ and $B$. Moreover, \eqref{eq:map_diamond_L(T)} is induced by the composition maps
\begin{equation*}
\compdg{\cat B}(T_A, h_B) \otimes \compdg{\cat B}(T_{A'}, T_A) \to \compdg{\cat B}(T_{A'}, h_B),
\end{equation*}
natural in $A'$ and $B$. In a similar fashion as for \eqref{eq:map_T_diamond}, we get a candidate counit morphism:
\begin{equation} \label{eq:T_L(T)_counit}
\varepsilon \colon T \diamond \ldual(T) \to h_{\cat B}.
\end{equation}
indeed, this morphism is induced by the maps:
\begin{align*}
\compdg{\cat B}(T_A, h_{B'}) \otimes T_A^B & \to h^B_{B'}, \\
\varphi \otimes x & \mapsto \varphi^B(x),
\end{align*}
natural in $B$ and $B'$. Also, we have the morphism
\begin{equation*}
t \colon h_{\cat A} \to \compdg{\cat B}(T_-, T_-),
\end{equation*}
induced by the action of $T$ on morphisms of $\cat A$. The following result tells us that the adjunction $T \dashv \ldual(T)$ is not very far from being obtained. 
\begin{lemma} \label{lemma:quasiadj_T_L(T)}
The diagram
\begin{equation}
\begin{gathered}
\xymatrix{
T \ar[r]^-{\sim} &  T \diamond h_{\cat A} \ar[r]^-{T \diamond t} & T \diamond \compdg{\cat B}(T_-, T_-) \ar[r]^-e & T \\
& T \diamond (\ldual(T) \diamond T) \ar[ur]^{T \diamond n} \ar[r]^\sim & (T \diamond \ldual(T)) \diamond T \ar[r]^-{\varepsilon \diamond T} & h_{\cat B} \diamond T \ar[u]_\sim
}
\end{gathered}
\end{equation}
is commutative, and the top row composition is the identity $1_T$.

Analogously, the diagram
\begin{equation}
\begin{gathered}
\xymatrix{
\ldual(T) \ar[r]^-{\sim} &  h_{\cat A} \diamond \ldual(T) \ar[r]^-{t \diamond \ldual(T)} & \compdg{\cat B}(T_-, T_-) \diamond \ldual(T) \ar[r]^-{e'} & \ldual(T) \\
& (\ldual(T) \diamond T) \diamond \ldual(T) \ar[ur]^{n \diamond \ldual(T)} \ar[r]^\sim & \ldual(T) \diamond (T \diamond \ldual(T)) \ar[r]^-{\ldual(T) \diamond \varepsilon} & \ldual(T) \diamond h_{\cat B} \ar[u]_-\sim
}
\end{gathered}
\end{equation}
is commutative, and the top row composition is the identity $1_{\ldual(T)}$.
\end{lemma}
\begin{proof}
They are all direct computations, which we leave to the reader.
\end{proof}
Taking h-projective resolutions of $T$ and of $\lder \ldual(T)$, and projecting every morphism in the derived category, we get the derived version of the above lemma:
\begin{lemma} \label{lemma:der_quasiadj_T_L(T)}
The diagram
\begin{equation}
\begin{gathered}
\xymatrix{
T   \ar[r]^-{\sim} &  T \ldiamond h_{\cat A} \ar[r]^-{T \ldiamond t} & T  \ldiamond \compdg{\cat B}(Q(T)_-, Q(T)_-) \ar[r]^-e & T  \\
& T \ldiamond (\lder \ldual(T) \ldiamond T) \ar[ur]^{T \ldiamond n} \ar[r]^\sim & (T \ldiamond \lder \ldual(T)) \ldiamond T \ar[r]^-{\varepsilon \ldiamond T} & h_{\cat B} \ldiamond T \ar[u]_\sim
}
\end{gathered}
\end{equation}
is commutative, and the top row composition is the identity $1_T$.

Analogously, the diagram
\begin{equation}
\begin{gathered}
\xymatrix{
\lder \ldual(T) \qis  h_{\cat A} \ldiamond \lder \ldual(T) \ar[r]^-{t \ldiamond \ldual(T)} & \compdg{\cat B}(Q(T)_-, Q(T)_-) \ldiamond \lder \ldual(T) \ar[r]^-{e'} & \lder \ldual(T) \\
(\lder \ldual(T) \ldiamond T) \ldiamond \lder \ldual(T) \ar[ur]^{n \ldiamond \lder \ldual(T)} \ar[r]^\sim & \lder \ldual(T) \ldiamond (T \ldiamond \lder \ldual(T)) \ar[r]^-{\lder \ldual(T) \ldiamond \varepsilon} & \lder \ldual(T) \ldiamond h_{\cat B} \ar[u]_-\sim
}
\end{gathered}
\end{equation}
is commutative, and the top row composition is the identity $1_{\lder \ldual(T)}$.
\end{lemma}
\begin{proof}
It follows immediately from Lemma \ref{lemma:quasiadj_T_L(T)}. Remember to compose with the h-projective resolutions $T \qis Q(T)$ and $\lder \ldual(T) \qis Q(\lder \ldual(T))$, at the start and the end of the top rows of both diagrams.
\end{proof}
Now, we see that the obstruction for $\ldual(T)$ to be adjoint to $T$ lies in the morphism \eqref{eq:morph_quasiunit} (or its derived version). For instance, if it is a natural isomorphism, then we may define the unit morphism
\begin{equation*}
\eta = n^{-1}t \colon h_{\cat A} \to \ldual(T) \diamond T,
\end{equation*}
and Lemma \ref{lemma:quasiadj_T_L(T)} tells us immediately that $T \dashv \ldual(T)$. Analogously, if the derived morphism $n \colon \lder \ldual(T) \ldiamond T \to \compdg{\cat B}(Q(T)_-, Q(T)_-)$ is a quasi-isomorphism, then we have a unit morphism $\eta$ in the derived category, and Lemma \ref{lemma:der_quasiadj_T_L(T)} tells us that $T \dashv \lder \ldual(T)$. A sufficient condition for $n$ to be (in some sense) invertible is actually the right (quasi-)representability of $T$:
\begin{prop} \label{prop:morph_quasiunit_iso}
If $T \in \compdg{\cat A, \cat B}$ is right representable, then \eqref{eq:morph_quasiunit} is an isomorphism. If it is right homotopy representable, then it is a homotopy equivalence.

If $T \in \dercomp{\cat A, \cat B}$ is right quasi-representable, then \eqref{eq:morph_quasiunit} induces a quasi-isomorphism (that is, the derived map
\begin{equation*}
n \colon \lder \ldual(T) \ldiamond T \to \compdg{\cat B}(Q(T)_-, Q(T)_-)
\end{equation*}
is a quasi-isomorphism).
\end{prop}
\begin{proof}
Assume that $T_A \cong h_{F(A)}$ or $T_A \approx h_{F(A)}$ for all $A \in \cat A$. Then, we have a commutative diagram:
\begin{equation*}
\xymatrix{
\displaystyle \int^B T^B_{A'} \otimes \compdg{\cat B}(T_A, h_B) \ar[r]^-n \ar[d]^\approx & \compdg{\cat B}(T_A, T_{A'}) \ar[d]^\approx \\
\displaystyle \int^B h^B_{F(A')} \otimes \compdg{\cat B}(h_{F(A)}, h_B) \ar[d]^\sim & \compdg{\cat B}(h_{F(A)}, h_{F(A')}) \ar[d]^\sim \\
\displaystyle \int^B h^B_{F(A')} \otimes h^{F(A)}_B \ar[r]^-\sim & h^{F(A)}_{F(A')}.
}
\end{equation*}
The lower vertical arrows, labeled with $\sim$, are given by the Yoneda lemma; the lower horizontal arrow is the co-Yoneda isomorphism. By dg-functoriality, the upper vertical arrows, labeled with $\approx$, are isomorphisms if $T$ is right representable, homotopy equivalences if $T$ is homotopy right representable. So, in the first case, $n$ is an isomorphism, and in the other case $n$ is a homotopy equivalence.

In the derived setting, just replace $T$ with its h-projective resolution $Q(T)$. Then, since $Q(T)_A$ and $h_{F(A)}$ are h-projective for all $A$, the quasi-isomorphism $Q(T)_A \qis h_{F(A)}$ is actually given by a homotopy equivalence; we apply the above argument and conclude that $n$ is a quasi-isomorphism, when viewed in the derived category.
\end{proof}
\begin{coroll} \label{coroll:quasirep_adjoints}
If $T \in \compdg{\cat A, \cat B}$ is right representable, then there is an adjunction $T \dashv \ldual(T)$ in $\kat{Bimod}$. If $T \in \dercomp{\cat A, \cat B}$ is right quasi-representable, then there is an adjunction $T \dashv \lder \ldual(T)$ in $\kat{DBimod}$.

Moreover, if $S \in \compdg{\cat B, \cat A}$ is left representable, then there is an adjunction $\rdual(S) \dashv S$ in $\kat{Bimod}$. If $S \in \dercomp{\cat B, \cat A}$ is left quasi-representable, then there is an adjunction $\rder \rdual(S) \dashv S$ in $\kat{DBimod}$.
\end{coroll}
\begin{proof}
The first part of the assertion follows directly from Proposition \ref{prop:morph_quasiunit_iso} and the above discussion. The second part is a consequence of Lemma \ref{lemma:repr_Isbell_equiv} and Proposition \ref{prop:derived_repr_Isbell_equiv}. Indeed, if $S$ is left representable, then $\rdual(S)$ is right representable, so we have $\rdual(S) \dashv \ldual \rdual(S)$, but $\ldual \rdual(S) \cong S$, and we are done. A similar argument in the derived setting shows that $\rder \rdual(S) \dashv S$.
\end{proof}

\section{Quasi-functors} \label{section:quasi-functors}
%Let $T \in \compdg{\cat A, \cat B}$ be a right representable bimodule, and assume $T_A \cong h_{F(A)}$ for all $A \in \cat A$. Then, there is a (unique) way to define a dg-functor $F \colon \cat A \to \cat B$ such that the above isomorphisms are natural in $A$:
%\begin{equation*}
%\xymatrix{
%T_A \ar[r]^-\sim \ar[d]^{T_f} & h_{F(A)} \ar[d]^{h_{F(f)}} \\
%T_{A'} \ar[r]^-\sim & h_{F(A')}.
%}
%\end{equation*}
% However, if $T \in \dercomp{\cat A, \cat B}$ is right quasi-representable, $T_A \qis h_{F(A)}$, then the above technique fails. Indeed, for all $f \colon A \to A'$ in $\cat A$, consider the following diagram:
%\begin{equation*}
%\xymatrix{
%T_A \ar[r]^-\qis \ar[d]^{T_f} & h_{F(A)} \ar@{.>}[d] \\
%T_{A'} \ar[r]^-\qis & h_{F(A')}
%}
%\end{equation*}
%For simplicity, assume that $T_A$ is h-projective for all $A$, so that the above horizontal quasi-isomorphisms are actually homotopy equivalences. Then, we are led to define $h_{F(f)} \colon h_{F(A)} \to h_{F(A')}$ by choosing a weak inverse of $T_A \approx h_{F(A)}$ and composing with $T_f$ and a representative of $T_{A'} \approx h_{F(A')}$. This gives us an arrow $F(f) \colon F(A) \to F(A')$, but this arrow is not uniquely determined (it is just ``unique up to homotopy''), so we are unable to obtain a dg-functor from this, even if we could in fact show that our attempt to define $F$ gives a ``weak (homotopy coherent) dg-functor''. In fact, right quasi-representable bimodules are themselves higher categorical gadgets: they are called \emph{quasi-functors}. 

Let $\cat A$ and $\cat B$ be dg-categories. By definition, a \emph{quasi-functor} $T \colon \cat A \to \cat B$ is an element of $\rqrep{\cat A, \cat B}$, namely, a right quasi-representable $\cat A$-$\cat B$-bimodule (assuming $\cat A$ is h-projective). The category of quasi-functors is usually denoted by $\mathrm{rep}(\cat A, \cat B)$ in literature (see, for instance, \cite{keller-dgcat}); in order to avoid confusion, we will stick to our (non standard) notation. The composition $\ldiamond$ descends to quasi-functors, namely, if $T \in \rqrep{\cat A, \cat B}$ and $S \in \rqrep{\cat B, \cat C}$, then $S \ldiamond T \in \rqrep{\cat A, \cat C}$. Indeed, assume that $T_A \qis h_{F(A)}$ and $S_B \qis h_{G(B)}$ for all $A \in \cat A$ and $B \in \cat B$. Then:
\begin{equation} \label{eq:quasifun_composition}
(S \ldiamond T)_A = \int^B Q(T)^B_A \otimes Q(S)_B \approx \int^B h^B_{F(A)} \otimes h_{G(B)} \cong h_{G(F(A))},
\end{equation}
where the last isomorphism follows by co-Yoneda lemma. It is also worth remarking that any dg-functor $F \colon \cat A \to \cat B$ can be identified with a quasi-functor, namely, the bimodule $h_F$. 
\subsection{Adjoints}
The results of Section \ref{section:bicat_bimod_adj} allow us to give a simple working characterisation of adjunctions of quasi-functors. Given quasi-functors $T \colon \cat A \to \cat B$ and $S \colon \cat B \to \cat A$, we say that \emph{$T$ is left adjoint to $S$ (and $S$ is right adjoint to $T$)} simply if $T \dashv S$ as bimodules, that is, in the bicategory $\kat{DBimod}$. Now, since $T$ is a quasi-functor, then we have the adjunction $T \dashv \lder \ldual(T)$, and so $S \qis \lder \ldual(T)$ (adjoints are always unique up to isomorphism). In particular, $S$ is left quasi-representable, and we have the adjunction $\rder \rdual(S) \dashv S$, so we also deduce that $T \qis \rder \rdual(S)$. In conclusion, we get the following result:
\begin{prop} \label{prop:quasifun_adjoint_repr}
A quasi-functor $T \colon \cat A \to \cat B$ has a left adjoint if and only if it is left quasi-representable. Moreover, it has a right adjoint if and only if $\lder \ldual(T)$ is right quasi-representable.
\end{prop}
There are sufficient hypotheses on the dg-categories that guarantee the existence of adjoints. They are, in some sense, particular finiteness conditions:
\begin{defin}
Let $\cat A$ be a dg-category. We say that $\cat A$ is \emph{locally perfect} if $\cat A(A,A')$ is a perfect $\basering k$-module: $\cat A(A,A') \in \per{\basering k}$ for all $A, A' \in \cat A$. We say that $\cat A$ is \emph{smooth} if the diagonal bimodule is perfect: $h_{\cat A} \in \per{\cat A \lotimes \opp{\cat A}}$.
\end{defin}
%\begin{remark} \label{remark:smooth_iff_perfect}
%It is worth mentioning that a dg-category $\cat A$ is smooth (resp. locally perfect) if and only $\perdg{\cat A}$ is smooth (resp. locally perfect): see \cite[Lemma 2.6]{toen-vaquie-moduli}.
%\end{remark}
We need a result adapted from \cite[Lemma 2.8]{toen-vaquie-moduli}:
\begin{lemma} \label{lemma:smooth_locallyperf}
Let $\cat A, \cat B$ be dg-categories, and let $T \in \dercomp{\cat A, \cat B}$ be a bimodule. If $\cat A$ is locally perfect and $T$ is a perfect bimodule, then $T_A$ is a perfect right $\cat B$-module for all $A \in \cat A$. Analogously, if $\cat B$ is locally perfect and $T$ is perfect, then $T^B$ is a perfect left $\cat A$-module for all $B \in \cat B$.

Conversely, if $\cat A$ is smooth and $T_A$ is perfect for all $A \in \cat A$, then $T$ is a perfect bimodule. Analogously, if $\cat B$ is smooth and $T^B$ is perfect for all $B \in \cat B$, then $T$ is perfect.
\end{lemma}
Now, we are able to prove the existence result of adjoints of quasi-functors:
\begin{thm} \label{thm:quasifun_adjoint_existence}
Let $\cat A, \cat B$ be dg-categories. Assume that $\cat A$ is triangulated and smooth, and that $\cat B$ is locally perfect. Let $T \colon \cat A \to \cat B$ be a quasi-functor. Then, $T$ admits both a left and a right adjoint.
\end{thm}
\begin{proof}
By hypothesis, $T_A$ is quasi-representable, in particular perfect, for all $A \in \cat A$. So, by Lemma \ref{lemma:smooth_locallyperf}, $T$ is a perfect bimodule. Since $\cat B$ is locally perfect, then $T^B$ is a perfect left $\cat A$-module for all $B \in \cat B$. But $\cat A$ is triangulated, so we conclude that $T^B$ is quasi-representable for all $B \in \cat B$, hence we conclude that $T$ has a left adjoint, by Proposition \ref{prop:quasifun_adjoint_repr}.

To prove the existence of the right adjoint, we apply a similar argument to $\lder \ldual(T)$. Since $T$ is right quasi-representable, then $\lder \ldual(T)$ is left quasi-representable, that is, $\lder \ldual(T)^A$ is quasi-representable, in particular perfect, for all $A \in \cat A$. Since $\cat A$ is smooth, we have that $\lder \ldual(T)$ is a perfect bimodule; since $\cat B$ is locally perfect, we deduce that $\lder \ldual(T)_B$ is a perfect right $\cat A$-module for all $B \in \cat B$. Since $\cat A$ is triangulated,  $\lder \ldual(T)_B$ is quasi-representable for all $B \in \cat B$, so in the end $\lder \ldual(T)$ is both left and right quasi-representable, and by Proposition \ref{prop:quasifun_adjoint_repr} we conclude that $T$ has a right adjoint.
\end{proof}
\begin{remark}
The above result is mentioned in \cite{toen-vaquie-saturated}, under stronger assumptions on the dg-categories, namely, saturatedness: see \cite[Definition 2.4]{toen-vaquie-moduli}. If a dg-category $\cat A$ is saturated, then in particular it is triangulated and $H^0(\cat A)$ is saturated as a triangulated category (\cite[Appendix A]{toen-algebrisation}), that is, any covariant or contravariant cohomological functor $H^0(\cat A) \to \Mod{\basering k}$ of finite type is representable. It is an easy exercise to show that exact functors between saturated (and Ext-finite) triangulated categories admit adjoints: Theorem \ref{thm:quasifun_adjoint_existence} can hence be viewed as an enhancement of this result in the dg framework.
\end{remark}

\bibliographystyle{alpha}
\bibliography{biblio}

\def\cprime{$'$}
\begin{thebibliography}{Woo82}

\bibitem[BL94]{bernstein-lunts-equivariant}
Joseph Bernstein and Valery Lunts.
\newblock {\em Equivariant sheaves and functors}, volume 1578 of {\em Lecture
  Notes in Mathematics}.
\newblock Springer-Verlag, Berlin, 1994.

\bibitem[CS15]{canonaco-stellari-internalhoms}
Alberto Canonaco and Paolo Stellari.
\newblock Internal homs via extensions of dg functors.
\newblock {\em Advances in Mathematics}, 277:100--123, 2015.

\bibitem[Dub70]{dubuc-enriched}
Eduardo~J. Dubuc.
\newblock {\em Kan extensions in enriched category theory}.
\newblock Lecture Notes in Mathematics, Vol. 145. Springer-Verlag, Berlin-New
  York, 1970.

\bibitem[JM89]{johnstone-moerdijk-toposes}
Peter~T. Johnstone and Ieke Moerdijk.
\newblock Local maps of toposes.
\newblock {\em Proc. London Math. Soc. (3)}, 58(2):281--305, 1989.

\bibitem[Kel05]{kelly-enriched}
Gregory~M. Kelly.
\newblock Basic concepts of enriched category theory.
\newblock {\em Repr. Theory Appl. Categ.}, (10):vi+137, 2005.
\newblock Reprint of the 1982 original [Cambridge Univ. Press, Cambridge;
  MR0651714].

\bibitem[Kel06]{keller-dgcat}
Bernhard Keller.
\newblock On differential graded categories.
\newblock In {\em International {C}ongress of {M}athematicians. {V}ol. {II}},
  pages 151--190. Eur. Math. Soc., Z\"urich, 2006.

\bibitem[Lor15]{fosco-coend}
Fosco Loregian.
\newblock This is the (co)end, my only (co)friend.
\newblock arXiv:1501.02503 [math.CT], 2015.

\bibitem[Rie14]{riehl-cathpy}
Emily Riehl.
\newblock {\em Categorical Homotopy Theory}.
\newblock Cambridge University Press, 2014.
\newblock Cambridge Books Online.

\bibitem[Tab05]{tabuada-dgcat}
Gon{\c{c}}alo Tabuada.
\newblock Une structure de cat\'egorie de mod\`eles de {Q}uillen sur la
  cat\'egorie des dg-cat\'egories.
\newblock {\em C. R. Math. Acad. Sci. Paris}, 340(1):15--19, 2005.

\bibitem[To{\"e}07]{toen-morita}
Bertrand To{\"e}n.
\newblock The homotopy theory of {$dg$}-categories and derived {M}orita theory.
\newblock {\em Invent. Math.}, 167(3):615--667, 2007.

\bibitem[TV07]{toen-vaquie-moduli}
Bertrand Toën and Michel Vaquié.
\newblock Moduli of objects in dg-categories.
\newblock {\em Annales Scientifiques de l’École Normale Supérieure},
  40(3):387--444, 2007.

\bibitem[TV08]{toen-algebrisation}
Bertrand To{\"e}n and Michel Vaqui{\'e}.
\newblock Alg\'ebrisation des vari\'et\'es analytiques complexes et
  cat\'egories d\'eriv\'ees.
\newblock {\em Math. Ann.}, 342(4):789--831, 2008.

\bibitem[TV15]{toen-vaquie-saturated}
Bertrand Toën and Michel Vaquié.
\newblock Systèmes de points dans les dg-catégories saturées.
\newblock arXiv:1504.07748 [math.AG], 2015.

\bibitem[Woo82]{wood-totalcat}
Richard~J. Wood.
\newblock Some remarks on total categories.
\newblock {\em Journal of Algebra}, 75(2):538--545, 1982.

\end{thebibliography}
\end{document}